\documentclass[12pt, reqno]{amsart}
\usepackage{amssymb}
\usepackage{eucal}
\usepackage{amsmath}
\usepackage{amscd}
\usepackage[dvips]{color}
\usepackage{multicol}
\usepackage[all]{xy}           %xypic macro for latex2.09
\usepackage{graphicx}
\usepackage{color}
\usepackage{colordvi}
\usepackage{xspace}
\usepackage{tikz}

\usepackage{ifpdf}
\ifpdf
 \usepackage[colorlinks,final,backref=page,hyperindex]{hyperref}
\else
 \usepackage[colorlinks,final,backref=page,hyperindex,hypertex]{hyperref}
\fi

\begin{document}
\newcommand {\emptycomment}[1]{} %to remove paragraphs

\newcommand{\nc}{\newcommand}
\newcommand{\delete}[1]{}
\nc{\mfootnote}[1]{\footnote{#1}} % Use this to show footnotes
\nc{\todo}[1]{\tred{To do:} #1}

%\delete{
\nc{\mlabel}[1]{\label{#1}}  % Use this to suppress names
\nc{\mcite}[1]{\cite{#1}}  % Use this to suppress names
\nc{\mref}[1]{\ref{#1}}  % Use this to suppress names
\nc{\meqref}[1]{\eqref{#1}} % Use this to suppress names
\nc{\mbibitem}[1]{\bibitem{#1}} % Use this to show number
%}

\delete{
\nc{\mlabel}[1]{\label{#1}  % Use the next two lines to show names
{\hfill \hspace{1cm}{\bf{{\ }\hfill(#1)}}}}
\nc{\mcite}[1]{\cite{#1}{{\bf{{\ }(#1)}}}}  % Use this lines to show names
\nc{\mref}[1]{\ref{#1}{{\bf{{\ }(#1)}}}}  % Use this lines to show names
\nc{\meqref}[1]{\eqref{#1}{{\bf{{\ }(#1)}}}} % Use this lines to show names
\nc{\mbibitem}[1]{\bibitem[\bf #1]{#1}} % Use this to show name
}

%%%%%%%%%%%%%%%%%%%%%%%% Statements
\newtheorem{thm}{Theorem}[section]
\newtheorem{lem}[thm]{Lemma}
\newtheorem{cor}[thm]{Corollary}
\newtheorem{pro}[thm]{Proposition}
\theoremstyle{definition}
\newtheorem{defi}[thm]{Definition}
\newtheorem{ex}[thm]{Example}
\newtheorem{rmk}[thm]{Remark}
\newtheorem{pdef}[thm]{Proposition-Definition}
\newtheorem{condition}[thm]{Condition}

\renewcommand{\labelenumi}{{\rm(\alph{enumi})}}
\renewcommand{\theenumi}{\alph{enumi}}
\renewcommand{\labelenumii}{{\rm(\roman{enumii})}}
\renewcommand{\theenumii}{\roman{enumii}}

\nc{\tred}[1]{\textcolor{red}{#1}}
\nc{\tblue}[1]{\textcolor{blue}{#1}}
\nc{\tgreen}[1]{\textcolor{green}{#1}}
\nc{\tpurple}[1]{\textcolor{purple}{#1}}
\nc{\btred}[1]{\textcolor{red}{\bf #1}}
\nc{\btblue}[1]{\textcolor{blue}{\bf #1}}
\nc{\btgreen}[1]{\textcolor{green}{\bf #1}}
\nc{\btpurple}[1]{\textcolor{purple}{\bf #1}}

\nc{\ts}[1]{\textcolor{purple}{Tianshui:#1}}
\nc{\cm}[1]{\textcolor{red}{Chengming:#1}}
\nc{\li}[1]{\textcolor{blue}{#1}}
\nc{\lir}[1]{\textcolor{blue}{Li:#1}}

%%%%%%%%%%%%%% Matrix symbols.

\nc{\twovec}[2]{\left(\begin{array}{c} #1 \\ #2\end{array} \right )}
\nc{\threevec}[3]{\left(\begin{array}{c} #1 \\ #2 \\ #3 \end{array}\right )}
\nc{\twomatrix}[4]{\left(\begin{array}{cc} #1 & #2\\ #3 & #4 \end{array} \right)}
\nc{\threematrix}[9]{{\left(\begin{matrix} #1 & #2 & #3\\ #4 & #5 & #6 \\ #7 & #8 & #9 \end{matrix} \right)}}
\nc{\twodet}[4]{\left|\begin{array}{cc} #1 & #2\\ #3 & #4 \end{array} \right|}

\nc{\mto}{\to }

\nc{\rk}{\mathrm{r}}
\newcommand{\g}{\mathfrak g}
\newcommand{\h}{\mathfrak h}
\newcommand{\pf}{\noindent{$Proof$.}\ }
\newcommand{\frkg}{\mathfrak g}
\newcommand{\frkh}{\mathfrak h}
\newcommand{\Id}{\rm{Id}}
\newcommand{\gl}{\mathfrak {gl}}
\newcommand{\ad}{\mathrm{ad}}
\newcommand{\add}{\frka\frkd}
\newcommand{\frka}{\mathfrak a}
\newcommand{\frkb}{\mathfrak b}
\newcommand{\frkc}{\mathfrak c}
\newcommand{\frkd}{\mathfrak d}
\newcommand {\comment}[1]{{\marginpar{*}\scriptsize\textbf{Comments:} #1}}
%%%%%%%%%%%%%%%%%%%%%%% symbols

\nc{\tforall}{\text{ for all }}

\nc{\svec}[2]{{\tiny\left(\begin{matrix}#1\\
#2\end{matrix}\right)\,}}  % column vector
\nc{\ssvec}[2]{{\tiny\left(\begin{matrix}#1\\
#2\end{matrix}\right)\,}} % subscript column vector

\nc{\bia}{{$\mathcal{P}$-bimodule ${\bf k}$-algebra}\xspace}
\nc{\bias}{{$\mathcal{P}$-bimodule ${\bf k}$-algebras}\xspace}

\nc{\OT}{constant $\theta$-}
\nc{\T}{$\theta$-}
\nc{\IT}{inverse $\theta$-}

%new commands for endo Lie

\nc{\mvarphi}{\psi}

%\nc{\endolie}{endoLiealgebra}
%\nc{\endolies}{endo Lie algebras\xspace}
%\nc{\Endolie}{Endo Lie algebra\xspace} \nc{\Endolies}{Endo Lie algebras\xspace}

\nc{\plhp}{endo pre-Lie algebra\xspace} \nc{\plhps}{endo pre-Lie algebras\xspace}
\nc{\weak}{coherent\xspace} % proper fair equitable balanced normal
\nc{\Weak}{Coherent\xspace}
\nc{\strong}{strong\xspace}

\nc{\cybe}{CYBE\xspace} \nc{\admset}{\{\pm x\}\cup K^{\times} x^{-1}}

\nc{\drep}{dual representation\xspace}
\nc{\dreping}{dually represents\xspace}
\nc{\dreped}{dually represented\xspace}

\nc{\LB}{\mathbf{LB}}
\nc{\MT}{\mathbf{MT}}
\nc{\OP}{\mathbf{OP}}
\nc{\PL}{\mathbf{PL}}
\nc{\MP}{\mathbf{MP}}

\nc{\opa}{\cdot_A}
\nc{\opb}{\cdot_B}

\nc{\pll}{\beta}
\nc{\plc}{\epsilon}

\nc{\adec}{\check{;}} \nc{\aop}{\alpha}
\nc{\dftimes}{\widetilde{\otimes}} \nc{\dfl}{\succ} \nc{\dfr}{\prec}
\nc{\dfc}{\circ} \nc{\dfb}{\bullet} \nc{\dft}{\star}
\nc{\dfcf}{{\mathbf k}} \nc{\apr}{\ast} \nc{\spr}{\cdot}
\nc{\twopr}{\circ} \nc{\tspr}{\star} \nc{\sempr}{\ast}
\nc{\disp}[1]{\displaystyle{#1}}
\nc{\bin}[2]{ (_{\stackrel{\scs{#1}}{\scs{#2}}})}  %binomial coeff
\nc{\binc}[2]{ \left (\!\! \begin{array}{c} \scs{#1}\\
    \scs{#2} \end{array}\!\! \right )}  %binomial coeff
\nc{\bincc}[2]{  \left ( {\scs{#1} \atop
    \vspace{-.5cm}\scs{#2}} \right )}  %binomial coeff
\nc{\sarray}[2]{\begin{array}{c}#1 \vspace{.1cm}\\ \hline
    \vspace{-.35cm} \\ #2 \end{array}}
\nc{\bs}{\bar{S}} \nc{\dcup}{\stackrel{\bullet}{\cup}}
\nc{\dbigcup}{\stackrel{\bullet}{\bigcup}} \nc{\etree}{\big |}
\nc{\la}{\longrightarrow} \nc{\fe}{\'{e}} \nc{\rar}{\rightarrow}
\nc{\dar}{\downarrow} \nc{\dap}[1]{\downarrow
\rlap{$\scriptstyle{#1}$}} \nc{\uap}[1]{\uparrow
\rlap{$\scriptstyle{#1}$}} \nc{\defeq}{\stackrel{\rm def}{=}}
\nc{\dis}[1]{\displaystyle{#1}} \nc{\dotcup}{\,
\displaystyle{\bigcup^\bullet}\ } \nc{\sdotcup}{\tiny{
\displaystyle{\bigcup^\bullet}\ }} \nc{\hcm}{\ \hat{,}\ }
\nc{\hcirc}{\hat{\circ}} \nc{\hts}{\hat{\shpr}}
\nc{\lts}{\stackrel{\leftarrow}{\shpr}}
\nc{\rts}{\stackrel{\rightarrow}{\shpr}} \nc{\lleft}{[}
\nc{\lright}{]} \nc{\uni}[1]{\tilde{#1}} \nc{\wor}[1]{\check{#1}}
\nc{\free}[1]{\bar{#1}} \nc{\den}[1]{\check{#1}} \nc{\lrpa}{\wr}
\nc{\curlyl}{\left \{ \begin{array}{c} {} \\ {} \end{array}
    \right .  \!\!\!\!\!\!\!}
\nc{\curlyr}{ \!\!\!\!\!\!\!
    \left . \begin{array}{c} {} \\ {} \end{array}
    \right \} }
\nc{\leaf}{\ell}       % number of leafs
\nc{\longmid}{\left | \begin{array}{c} {} \\ {} \end{array}
    \right . \!\!\!\!\!\!\!}
\nc{\ot}{\otimes} \nc{\sot}{{\scriptstyle{\ot}}}
\nc{\otm}{\overline{\ot}}
\nc{\ora}[1]{\stackrel{#1}{\rar}}
\nc{\ola}[1]{\stackrel{#1}{\la}}%${\Bbb Z}$
\nc{\pltree}{\calt^\pl}
\nc{\epltree}{\calt^{\pl,\NC}}
\nc{\rbpltree}{\calt^r}
\nc{\scs}[1]{\scriptstyle{#1}} \nc{\mrm}[1]{{\rm #1}}
\nc{\dirlim}{\displaystyle{\lim_{\longrightarrow}}\,}
\nc{\invlim}{\displaystyle{\lim_{\longleftarrow}}\,}
\nc{\mvp}{\vspace{0.5cm}} \nc{\svp}{\vspace{2cm}}
\nc{\vp}{\vspace{8cm}} \nc{\proofbegin}{\noindent{\bf Proof: }}
%\nc{\proofbegin}{\begin{proof}} % AMS command
\nc{\proofend}{$\blacksquare$ \vspace{0.5cm}}
%\nc{\proofend}{\end{proof}} %AMS command
\nc{\freerbpl}{{F^{\mathrm RBPL}}}
\nc{\sha}{{\mbox{\cyr X}}}  %used to be \cyr
\nc{\ncsha}{{\mbox{\cyr X}^{\mathrm NC}}} \nc{\ncshao}{{\mbox{\cyr
X}^{\mathrm NC,\,0}}}
\nc{\shpr}{\diamond}    %Shuffle product
\nc{\shprm}{\overline{\diamond}}    %Shuffle product
\nc{\shpro}{\diamond^0}    %Shuffle product
\nc{\shprr}{\diamond^r}     %product on controlled trees
\nc{\shpra}{\overline{\diamond}^r}
\nc{\shpru}{\check{\diamond}} \nc{\catpr}{\diamond_l}
\nc{\rcatpr}{\diamond_r} \nc{\lapr}{\diamond_a}
\nc{\sqcupm}{\ot}
\nc{\lepr}{\diamond_e} \nc{\vep}{\varepsilon} \nc{\labs}{\mid\!}
\nc{\rabs}{\!\mid} \nc{\hsha}{\widehat{\sha}}
\nc{\lsha}{\stackrel{\leftarrow}{\sha}}
\nc{\rsha}{\stackrel{\rightarrow}{\sha}} \nc{\lc}{\lfloor}
\nc{\rc}{\rfloor}
\nc{\tpr}{\sqcup}
\nc{\nctpr}{\vee}
\nc{\plpr}{\star}
\nc{\rbplpr}{\bar{\plpr}}
\nc{\sqmon}[1]{\langle #1\rangle}
\nc{\forest}{\calf}
\nc{\altx}{\Lambda_X} \nc{\vecT}{\vec{T}} \nc{\onetree}{\bullet}
\nc{\Ao}{\check{A}}
\nc{\seta}{\underline{\Ao}}
\nc{\deltaa}{\overline{\delta}}
\nc{\trho}{\tilde{\rho}}

\nc{\rpr}{\circ}
%\nc{\apr}{\cdot}
\nc{\dpr}{{\tiny\diamond}}
\nc{\rprpm}{{\rpr}}

%%%%%%%%%%%%%%%%%%%%% roman fonts, in alphabetic order
\nc{\mmbox}[1]{\mbox{\ #1\ }} \nc{\ann}{\mrm{ann}}
\nc{\Aut}{\mrm{Aut}} \nc{\can}{\mrm{can}}
\nc{\twoalg}{{two-sided algebra}\xspace}
\nc{\colim}{\mrm{colim}}
\nc{\Cont}{\mrm{Cont}} \nc{\rchar}{\mrm{char}}
\nc{\cok}{\mrm{coker}} \nc{\dtf}{{R-{\rm tf}}} \nc{\dtor}{{R-{\rm
tor}}}
\renewcommand{\det}{\mrm{det}}
\nc{\depth}{{\mrm d}}
\nc{\Div}{{\mrm Div}} \nc{\End}{\mrm{End}} \nc{\Ext}{\mrm{Ext}}
\nc{\Fil}{\mrm{Fil}} \nc{\Frob}{\mrm{Frob}} \nc{\Gal}{\mrm{Gal}}
\nc{\GL}{\mrm{GL}} \nc{\Hom}{\mrm{Hom}} \nc{\hsr}{\mrm{H}}
\nc{\hpol}{\mrm{HP}} \nc{\id}{\mrm{id}} \nc{\im}{\mrm{im}}
\nc{\incl}{\mrm{incl}} \nc{\length}{\mrm{length}}
\nc{\LR}{\mrm{LR}} \nc{\mchar}{\rm char} \nc{\NC}{\mrm{NC}}
\nc{\mpart}{\mrm{part}} \nc{\pl}{\mrm{PL}}
\nc{\ql}{{\QQ_\ell}} \nc{\qp}{{\QQ_p}}
\nc{\rank}{\mrm{rank}} \nc{\rba}{\rm{RBA }} \nc{\rbas}{\rm{RBAs }}
\nc{\rbpl}{\mrm{RBPL}}
\nc{\rbw}{\rm{RBW }} \nc{\rbws}{\rm{RBWs }} \nc{\rcot}{\mrm{cot}}
\nc{\rest}{\rm{controlled}\xspace}
\nc{\rdef}{\mrm{def}} \nc{\rdiv}{{\rm div}} \nc{\rtf}{{\rm tf}}
\nc{\rtor}{{\rm tor}} \nc{\res}{\mrm{res}} \nc{\SL}{\mrm{SL}}
\nc{\Spec}{\mrm{Spec}} \nc{\tor}{\mrm{tor}} \nc{\Tr}{\mrm{Tr}}
\nc{\mtr}{\mrm{sk}}

%%%%%%%%%%%%%%%%%% bold face
\nc{\ab}{\mathbf{Ab}} \nc{\Alg}{\mathbf{Alg}}
\nc{\Algo}{\mathbf{Alg}^0} \nc{\Bax}{\mathbf{Bax}}
\nc{\Baxo}{\mathbf{Bax}^0} \nc{\RB}{\mathbf{RB}}
\nc{\RBo}{\mathbf{RB}^0} \nc{\BRB}{\mathbf{RB}}
\nc{\Dend}{\mathbf{DD}} \nc{\bfk}{{\bf k}} \nc{\bfone}{{\bf 1}}
\nc{\base}[1]{{a_{#1}}} \nc{\detail}{\marginpar{\bf More detail}
    \noindent{\bf Need more detail!}
    \svp}
\nc{\Diff}{\mathbf{Diff}} \nc{\gap}{\marginpar{\bf
Incomplete}\noindent{\bf Incomplete!!}
    \svp}
\nc{\FMod}{\mathbf{FMod}} \nc{\mset}{\mathbf{MSet}}
\nc{\rb}{\mathrm{RB}} \nc{\Int}{\mathbf{Int}}
\nc{\Mon}{\mathbf{Mon}}
%\nc{\remark}{\noindent{\bf Remark: }}
\nc{\remarks}{\noindent{\bf Remarks: }}
\nc{\OS}{\mathbf{OS}} %free operated semigroup
\nc{\Rep}{\mathbf{Rep}}
\nc{\Rings}{\mathbf{Rings}} \nc{\Sets}{\mathbf{Sets}}
\nc{\DT}{\mathbf{DT}}

%%%%%%%%%%%%%%%%%%%Bbb fonts
\nc{\BA}{{\mathbb A}} \nc{\CC}{{\mathbb C}} \nc{\DD}{{\mathbb D}}
\nc{\EE}{{\mathbb E}} \nc{\FF}{{\mathbb F}} \nc{\GG}{{\mathbb G}}
\nc{\HH}{{\mathbb H}} \nc{\LL}{{\mathbb L}} \nc{\NN}{{\mathbb N}}
\nc{\QQ}{{\mathbb Q}} \nc{\RR}{{\mathbb R}} \nc{\BS}{{\mathbb{S}}} \nc{\TT}{{\mathbb T}}
\nc{\VV}{{\mathbb V}} \nc{\ZZ}{{\mathbb Z}}

%%%%%%%%%%%%%%%%%%% cal fonts

\nc{\calao}{{\mathcal A}} \nc{\cala}{{\mathcal A}}
\nc{\calc}{{\mathcal C}} \nc{\cald}{{\mathcal D}}
\nc{\cale}{{\mathcal E}} \nc{\calf}{{\mathcal F}}
\nc{\calfr}{{{\mathcal F}^{\,r}}} \nc{\calfo}{{\mathcal F}^0}
\nc{\calfro}{{\mathcal F}^{\,r,0}} \nc{\oF}{\overline{F}}
\nc{\calg}{{\mathcal G}} \nc{\calh}{{\mathcal H}}
\nc{\cali}{{\mathcal I}} \nc{\calj}{{\mathcal J}}
\nc{\call}{{\mathcal L}} \nc{\calm}{{\mathcal M}}
\nc{\caln}{{\mathcal N}} \nc{\calo}{{\mathcal O}}
\nc{\calp}{{\mathcal P}} \nc{\calq}{{\mathcal Q}} \nc{\calr}{{\mathcal R}}
\nc{\calt}{{\mathcal T}} \nc{\caltr}{{\mathcal T}^{\,r}}
\nc{\calu}{{\mathcal U}} \nc{\calv}{{\mathcal V}}
\nc{\calw}{{\mathcal W}} \nc{\calx}{{\mathcal X}}
\nc{\CA}{\mathcal{A}}

%%%%%%%%%%%%%%%%%%  frak fonts
\nc{\fraka}{{\mathfrak a}} \nc{\frakB}{{\mathfrak B}}
\nc{\frakb}{{\mathfrak b}} \nc{\frakd}{{\mathfrak d}}
\nc{\oD}{\overline{D}}
\nc{\frakF}{{\mathfrak F}} \nc{\frakg}{{\mathfrak g}}
\nc{\frakm}{{\mathfrak m}} \nc{\frakM}{{\mathfrak M}}
\nc{\frakMo}{{\mathfrak M}^0} \nc{\frakp}{{\mathfrak p}}
\nc{\frakS}{{\mathfrak S}} \nc{\frakSo}{{\mathfrak S}^0}
\nc{\fraks}{{\mathfrak s}} \nc{\os}{\overline{\fraks}}
\nc{\frakT}{{\mathfrak T}}
\nc{\oT}{\overline{T}}
%\nc{\frakx}{{\mathfrak x}}
\nc{\frakX}{{\mathfrak X}} \nc{\frakXo}{{\mathfrak X}^0}
\nc{\frakx}{{\mathbf x}}
%\nc{\frakTxo}{{\frakTx}^0}
\nc{\frakTx}{\frakT}      %All rooted trees, correspond to \ncsha(X)
\nc{\frakTa}{\frakT^a}        % rooted trees for \ncsha(A)
\nc{\frakTxo}{\frakTx^0}   % rooted trees for \ncshao(X)
\nc{\caltao}{\calt^{a,0}}   % rooted trees for \ncshao(A)
\nc{\ox}{\overline{\frakx}} \nc{\fraky}{{\mathfrak y}}
\nc{\frakz}{{\mathfrak z}} \nc{\oX}{\overline{X}}

\font\cyr=wncyr10

\nc{\al}{\alpha}
\nc{\lam}{\lambda}
\nc{\lr}{\longrightarrow}

%%%%%%%%%%%%%%%%%%%%%%%%%%%%%%%%%%%%%%%%%%%%%%%%%%%%%%%%%%%%%%%%%%

\title[Categories for Lie bialgebras and related structures]{Coherent categorical structures for Lie bialgebras, Manin triples, classical $r$-matrices and pre-Lie algebras}

\author{Chengming Bai}
\address{Chern Institute of Mathematics \& LPMC, Nankai University, Tianjin 300071, China}
         \email{baicm@nankai.edu.cn}

\author{Li Guo}
\address{Department of Mathematics and Computer Science, Rutgers University, Newark, NJ 07102, USA}
         \email{liguo@rutgers.edu}

\author{Yunhe Sheng}
\address{Department of Mathematics, Jilin University, Changchun 130012, Jilin, China}
\email{shengyh@jlu.edu.cn}

\date{\today}

\begin{abstract}
The broadly applied notions of Lie bialgebras, Manin triples, classical $r$-matrices and $\mathcal{O}$-operators of Lie algebras owe their importance to the close relationship among them.
Yet these notions and their correspondences are mostly understood as classes of objects and maps among the classes.
To gain categorical insight, this paper introduces, for each of the classes, a notion of homomorphisms, uniformly called coherent homomorphisms, so that the classes of objects become categories and the maps among the classes become functors or category equivalences. For this purpose, we start with the notion of an endo Lie algebra, consisting of a Lie algebra equipped with a Lie algebra endomorphism. We then generalize the above classical notions for Lie algebras to endo Lie algebras. As a result, we obtain the notion of coherent endomorphisms for each of the classes, which then generalizes to the notion of coherent homomorphisms by a polarization process. The coherent homomorphisms are compatible with the correspondences among the various constructions, as well as with the category of pre-Lie algebras.
\end{abstract}

\subjclass[2010]{
%17A30,  %algebras satisfying other identities
%16T05,  %Hopf algebras and their applications
17B62,   %Lie bialgebras, Lie coalgebras
17B38,       %Yang-Baxter equations and Rota-Baxter operators,  
16T10,   %bialgebra
16T25,   %Yang-Baxter equation
12H10   %difference algebra
16W99,  %rings and algebras with additional structure/none of above, but in this section
57R56.   %Topological quantum field theories
%81R60   %noncommutative geometry
%81T45,   %Topological field theories
}

\keywords{Lie bialgebra, Manin triple, classical Yang-Baxter equation, $r$-matrix, $\calo$-operator, pre-Lie algebra.\\
Published online: Forum Mathematicum, 2022. https://doi.org/10.1515/forum-2021-0240
}

\maketitle

\vspace{-1.2cm}

\tableofcontents

\vspace{-.5cm}

\allowdisplaybreaks

\section{Introduction}
\vspace{-.2cm}
This paper gives a categorical structure to each of the classes of Lie bialgebras, Manin triples, classical $r$-matrices, $\calo$-operators and pre-Lie algebras by introducing their morphisms. The morphisms are compatible with natural correspondences among these classes, so that these correspondences become functors.

\vspace{-.2cm}
\subsection{Lie bialgebras, Manin triples and the related structures}

The Lie bialgebra is the algebraic structure corresponding
to a Poisson-Lie group. It is also the classical structure of a quantized universal enveloping algebra~\mcite{CP,D}.

The great importance of Lie bialgebras is reflected by its
close relationship with several other fundamental notions.
First Lie bialgebras are
characterized by Manin triples and  matched pairs  of Lie algebras~\mcite{D1}.
In fact, there is a one-one correspondence between Lie
bialgebras and Manin triples of Lie algebras. The same holds for
Lie bialgebras and matched pairs  of Lie algebras associated to coadjoint representations.

Furthermore, solutions of the classical Yang-Baxter
equation, or the  classical  $r$-matrices, naturally give rise to coboundary
Lie bialgebras \cite{CP,STS}. Furthermore, such solutions are provided by
$\calo$-operators, which in turn are provided by pre-Lie algebras. See~\cite{Bai2,Bai3,Bu,Ku} for more details. These close relations can be summarized in the following diagram where some of the correspondences go both ways, showing by arrows in both directions, and some are even one-one correspondences, showing by two-directional double arrows.
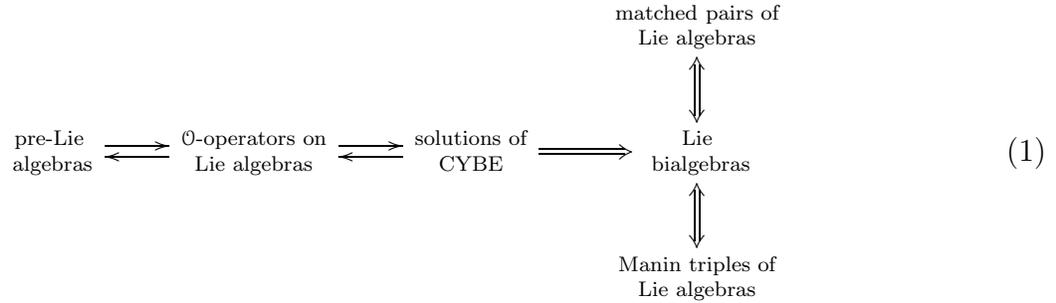
\vspace{-.1cm}
\begin{equation}
    \begin{split}
        \xymatrix{
            &&&\text{matched pairs of}\atop \text{Lie algebras} \\
            \text{pre-Lie}\atop\text{ algebras} \ar@<.4ex>[r]      & \mathcal{O}\text{-operators on}\atop\text{Lie algebras}\ar@<.4ex>[l] \ar@<.4ex>[r]&
             \text{solutions of}\atop \text{CYBE} \ar@<.4ex>[l]
            \ar@2{->}[r]& \text{Lie}\atop \text{ bialgebras} \ar@2{<->}[d] \ar@2{<->}[u] \\
            &&& \text{Manin triples of}\atop \text{Lie algebras} & }
    \end{split}
    \mlabel{eq:bigdiag}
\end{equation}
\vspace{-.6cm}
\subsection{Existing morphisms of the structures}
For further studies and applications of the important structures and their relations in the above diagram, it is fundamental to understand them in the context of categories.
The first step in this direction is to define suitable morphisms for these structures to make them into categories, so that their relations can be made precise as functors and equivalences among these categories. Unfortunately, the understanding of their morphisms is quite limited and can be summarized as follows.
\vspace{-.2cm}
\subsubsection{Morphisms of Lie bialgebras and Manin triples}
A notion of homomorphisms of Lie bialgebras has been defined in analog to homomorphisms of associative bialgebras and has been applied in the important quantization of Lie bialgebras~\mcite{CP,En,EK,EK1}.
The notion of isomorphisms of two Manin triples is also defined and agrees with the isomorphisms of Lie bialgebras.
However, the expected extension of isomorphisms of Manin triples
to homomorphisms do not agree with the homomorphisms of their
corresponding Lie bialgebras~\mcite{CP}.

Explicitly, let $f$ be a homomorphism between two
Lie bialgebras $(\g_1, \delta_1)$ and $(\g_2,\delta_2)$. Then
$f:\g_1\mto \g_2$ is a homomorphism of Lie algebras and
$f^*:\g_2^*\rightarrow \g_1^*$ is a homomorphism of the dual Lie
algebras. It is natural to expect that $f$ should also give a
homomorphism of the corresponding Manin triples $(\g_1\bowtie
\g_1^*,\g_1,\g_1^*)$ and $(\g_2\bowtie
\g_2^*,\g_2,\g_2^*)$.  In fact, as the most
natural choice, the map $f+f^*$ should be a homomorphism between the two Lie algebras $\g_1\bowtie
\g_1^*$ and $\g_2\bowtie
\g_2^*$. However, this does not hold, even in the
case when  $\g_1=\g_2$.
\vspace{-.1cm}
\subsubsection{Morphisms of classical $r$-matrices, $\calo$-operators and pre-Lie algebras}
Recently, a notion of weak homomorphisms of classical $r$-matrices was
introduced in~\mcite{TBGS} in associate with a notion of homomorphisms of the corresponding $\calo$-operators.
But this notion of weak homomorphisms is defined only for the skew-symmetric classical $r$-matrices (hence only for triangular Lie bialgebras). It is natural to ask what to expect without the skew-symmetric condition, so that the anticipated homomorphisms of classical $r$-matrices are compatible with a suitably defined homomorphism of Lie bialgebras.

Homomorphisms of pre-Lie algebras are naturally defined, but it is not known how they could be preserved under the correspondence  between  pre-Lie algebras and $\calo$-operators.
\vspace{-.2cm}
\subsection{Our approach}
Thus so far our understanding of the important classes of objects in \meqref{eq:bigdiag} and their relationship mostly remains at the level of sets and set correspondences. Morphisms for these objects are either not defined, or are defined but not compatible with the correspondences.

The goal of this paper is to introduce morphisms for all the classes in the diagram so that the diagram gives functors and natural equivalences of the resulting categories. Since the homomorphisms we will define for the various classes of objects are compatible with the correspondences among the classes, we use the uniform term of {\bf \weak homomorphisms} for the new homomorphisms in all the classes.

To reach our goal, we utilize two strategies. The first strategy is a
polarization process that allows us to first consider endomorphisms instead of homomorphisms. The second strategy is a change of order of the operations which equip Lie
algebras with two extra structures: bialgebras and
endomorphisms.
\vspace{-.1cm}
\subsubsection{Polarization}

Our strategy of polarization is in a sense similar to polarizing a homogeneous polynomial in one variable to a multilinear form in multivariable by linear substitutions and derivations~\mcite{Pr}. The depolarization is then evaluating along a diagonal line. For us depolarization of morphisms in a category gives endomorphisms. So for each of the classes of constructions such as Lie bialgebras, instead of attempting to define morphisms between any two of them, we focus on the special case of endomorphisms on any given object. Taking Lie bialgebras as an example, we first define endomorphisms of a Lie bialgebra, instead of homomorphisms between any two Lie bialgebras.

This depolarization does not reduce the level of difficulty by itself, but allows us to look the endomorphisms of a Lie bialgebra for instance from a different angle and apply the strategy of changing of operation orders.
\vspace{-.1cm}
\subsubsection{Endomorphisms of Lie algebras and bialgebras of Lie algebras with endomorphisms}

With the polarization reduction, our goal is to define endomorphisms of Lie bialgebras and the related structures in diagram \meqref{eq:bigdiag}. To define endomorphisms for a Lie bialgebra, we can regard it as equipping an extra endomorphism structure to the    Lie bialgebra which, by itself, is obtained from equipping a bialgebra structure to a Lie algebra. Thus we are looking at the composition of two processes of equipping two extra structures to a given structure, beginning with a Lie algebra.

Our second strategy is to switch the order of these two processes. This is shown in the following diagram for the instance of Lie bialgebras. Each of the other classes in Diagram~\meqref{eq:bigdiag} is also obtained from equipping a Lie algebra with an extra structure, so can be treated in the same way.
\vspace{-1cm}
\begin{equation}
    \begin{split}
        \xymatrix{
            \text{\small Lie algebras} \ar[rrr]^{\text{\small endomorphism}} \ar[d]^{\text{\small bialgebraization}} &&& \text{\small Endo Lie algebras} \ar[d]^{\text{\small bialgebraization}} \\
            \text{\small Lie bialgebras}\atop
            {=\text{\small Bi-Lie algebras}} \ar[rrr]^{\text{\small endomorphism}}
            &&& {\text{\small Bi-endo Lie algebras}} \atop \text{\small =Endo Lie bialgebras}
        }
    \end{split}
    \mlabel{eq:diag}
\end{equation}
\vspace{-.4cm}

In this diagram, the vertical arrows are equipping a given structure by a suitable bialgebra structure and the horizontal arrows are equipping a given structure with suitable endomorphisms.
Our goal for a notion of endomorphisms
of Lie bialgebras, which amounts to endomorphisms of the bialgebra structures on Lie algebras, is
starting from the Lie algebras in the diagram, going downward and then right.
Instead of attacking this apparently challenging task, we go first right and then downward, that is,  we first equip Lie algebras with endomorphisms, called {\bf endo Lie algebras}. We then attempt to equip a bialgebra structure for endo Lie algebras. The notion of an endo Lie algebra is interesting on its own right: an associative algebra, in particular a commutative associative algebra equipped with an injective endomorphism is called a difference algebra~\mcite{Co,Le} as an algebraic study of difference equations and a close analogy of differential algebras~\mcite{PS}.

The same strategy is applied to the other structures in Diagram~\meqref{eq:bigdiag}: instead of defining endomorphisms of a Lie algebra with the  various  extra structures in the diagram, we define the various extra structures on an endo Lie algebra, resulting in the following diagram.
\vspace{-.2cm}
\begin{equation}
    \begin{split}
        \xymatrix{
            &&&\text{matched pairs of}\atop \text{endo Lie algebras} \\
            \text{endo} \atop \text{pre-Lie algebras} \ar@<.4ex>[r]
            & \mathcal{O}\text{-operators on}\atop\text{endo Lie algebras}\ar@<.4ex>[l] \ar@<.4ex>[r]&
            \text{solutions of}\atop \text{endo CYBE} \ar@<.4ex>[l]
            \ar@2{->}^{}[r]& \text{bialgebras of}\atop \text{endo Lie algebras} \ar@2{<->}_{}[d] \ar@2{<->}^{}[u] \\
            &&& \text{Manin triples of}\atop \text{endo Lie algebras} & }
    \end{split}
    \mlabel{eq:bigdiagend}
\end{equation}
\vspace{-.4cm}

Once constructed, each of the structures on endo Lie algebras naturally gives rise to a notion of endomorphisms of this structure on Lie algebras. Once this is done, then as noted above, a polarization process extends the notion of endomorphisms to a notion of homomorphisms, equipping each class of objects with a category structure.
Furthermore, the correspondences among the classes are naturally functors and, in the case of one-one correspondences, equivalences. These are summarized in the following enriched diagram of \meqref{eq:bigdiag}. Here the labels on the arrows indicate the results where the functors and equivalences are established. It is remarkable that the natural categorical structure of pre-Lie algebras is compatible with the category of $\calo$-operators hereby introduced. In fact, there is a pair of adjoint functors between the two categories.
\vspace{-.6cm}
\begin{equation}
\mlabel{eq:bigdiagcat}
\begin{split}
 \xymatrix{
 &&&&&& {\begin{subarray}{c} \text{category of}\\ \text{matched pairs of}\\ \text{Lie algebras} \end{subarray}}\\
 {\begin{subarray}{c} \text{category of}\\ \text{pre-Lie} \\ \text{algebras}\end{subarray}}  \ar@<.3ex>[rr]^{\rm Prop~\mref{pro:prelieo}}_{\rm Prop~\mref{pro:otopl}}&& {\begin{subarray}{c} \text{category of } \\ \mathcal{O}\text{-operators on} \\ \text{Lie algebras} \end{subarray}}  \ar@<.3ex>[ll]\ar@<.3ex>[rr]^{\rm Cor~ \mref{cor:zero}}&&
{\begin{subarray}{c} \text{category of } \\ \text{solutions of}\\ \text{CYBE}
\end{subarray}} \ar@<.4ex>[ll]^{\rm {Thm}~\mref{pp:oprm}}
 \ar@2{->}^{\rm {Thm}~\mref{co:r12}}[rr]&& {\begin{subarray}{c} \text{category of} \\ \text{Lie bialgebras}\end{subarray}} \ar@2{<->}_{\rm Thm~\mref{thm:equivalence111}}[d] \ar@2{<->}^{\rm Thm~\mref{thm:equivalence111}}[u] \\
&&&&&& {\begin{subarray}{c} \text{category of} \\ \text{Manin triples of}\\ \text{Lie algebras}\end{subarray}}
}
\end{split}
\end{equation}
\vspace{-.5cm}

\subsection{Outline of the paper}

We first introduce in Section~\mref{sec:endolie} the notion of an endo Lie algebra and formulate a bialgebra theory for endo Lie algebras together with their equivalences to  matched pairs and Manin triples of endo Lie algebras.
We then interpret a bialgebra of endo Lie algebras   as an endomorphism  of a Lie bialgebra, which then naturally generalizes to a homomorphism of Lie bialgebras that is compatible with that of Manin triples (and  matched pairs) of Lie algebras, showing that the correspondences among Lie bialgebras, matched pairs of Lie algebras, and Manin triples of Lie algebras are equivalences of categories (Theorem~\mref{thm:equivalence111}).

We then extend in Section~\mref{sec:rmat} the classical relations
of Lie bialgebras with the classical Yang-Baxter
equation as well as  classical  $r$-matrices to the context of endo Lie
algebras. This naturally gives rise to a notion of \weak
homomorphisms for any $r$-matrices, not just the skew-symmetric
ones. This notion is shown to be compatible with the \weak
homomorphisms of Lie bialgebras, leading to a functor of the
corresponding categories (Theorem~\mref{co:r12}).

Finally in Section~\mref{sec:oop}, we give the notion of $\calo$-operators on endo Lie algebras and apply it to define \weak homomorphisms of $\calo$-operators in such a way that is compatible with the \weak homomorphisms of classical  $r$-matrices in Section~\mref{sec:rmat} (Theorem~\mref{pp:oprm} and Corollary~\mref{cor:zero}).
This notion of \weak homomorphisms of $\calo$-operators is moreover compatible with the natural homomorphism of pre-Lie algebras, giving rise to a pair of adjoint functors between the corresponding two categories (Propositions~\mref{pro:prelieo} and \mref{pro:otopl}). We also consider a case where all the constructions can be given explicitly, providing natural examples of \weak isomorphisms of Lie bialgebras that are not the previously defined isomorphisms.
This further justifies the significance of the \weak homomorphisms of Lie bialgebras introduced in this paper.

\smallskip

\noindent
{\bf Notations. } Throughout this paper, all vector spaces, tensor products, and linear homomorphisms are over a fixed field $K$.  Let $\End(V)$ denote the space of linear operators on   a vector space   $V$. The vector spaces and Lie algebras are finite dimensional unless otherwise specified.

\vspace{-.3cm}

\section{Endo Lie algebras and their bialgebras}
\mlabel{sec:endolie}
In this section we introduce the notion of endo Lie algebras and give the equivalent structures of bialgebras, matched pairs and Manin triples for endo Lie algebras.
\vspace{-.1cm}

\subsection{Endo Lie algebras and their representations}
\mlabel{sec:rep} We first introduce the notion of
a representation of an endo Lie algebra characterized by a semi-direct product. We
then introduce the notion of a \drep of an endo Lie algebra in order to construct a reasonable
representation on the dual space.

\begin{defi}
  An {\bf endo Lie algebra} is
a triple $(\frak g,[\;,\;],\phi)$, or simply $(\frak g, \phi)$, where $ (\frak g,[\;,\;])$ is a Lie algebra and  $\phi:\g \rightarrow \g$ is a Lie algebra endomorphism.
\end{defi}
As an analogy of a Lie algebra representation, we have
\begin{defi}
A {\bf representation} of an endo Lie algebra $(\frak g, \phi)$ is a triple $(V, \rho, \alpha)$ where $(V, \rho)$ is a representation of the Lie algebra $\g$ and $\alpha\in \End(V)$ such that
 \begin{equation}
    \alpha (\rho(x)(v))=\rho(\phi(x))(\alpha(v)),\;\;\forall x\in \g, v\in V.
    \mlabel{eq:repn}
\end{equation}Two representations $(V_1, \rho_1, \alpha_1)$ and $(V_2, \rho_2,
\alpha_2)$ of an endo Lie algebra  $(\g, \phi)$ are called {\bf equivalent} if
there exists a linear isomorphism $\varphi:V_1\rightarrow V_2$
such that
 \begin{equation}
\varphi (\rho_1(x)(v))=\rho_2(x)(\varphi(v)),\quad \varphi \alpha_1
(v)=\alpha_2\varphi(v),\quad \forall x\in \g, v\in V_1.
\mlabel{de:2.1} \end{equation}
\end{defi}
With $\ad$  denoting the adjoint representation of the  Lie algebra $\g$, the triple $(\g, \ad,\phi)$ is naturally a representation of the endo Lie algebra $(\g,\phi)$, called the {\bf adjoint representation} of $(\g,\phi)$.

For vector spaces $V_1$ and $V_2$, and linear maps
$\phi_1:V_1\rightarrow V_1$ and $\phi_2:V_2\rightarrow V_2$, we abbreviate $\phi_1+\phi_2$ for the linear map
\begin{equation}
\phi_{V_1\oplus V_2}: V_1\oplus V_2\mto V_1\oplus V_2,
\quad \phi_{V_1\oplus V_2}(v_1+v_2):=\phi_1(v_1)+\phi_2(v_2),\quad
\forall v_1\in V_1, v_2\in V_2. \mlabel{eq:Liehom}
\end{equation}

For a Lie algebra $\g$, a linear space $V$ and a linear map
$\rho:\g\rightarrow {\rm End}(V)$, define a multiplication
$[\cdot,\cdot]_\ltimes$ on $\g\oplus V$ by
\begin{equation}
[x+u,y+v]_\ltimes:=[x,y]+\rho(x)v-\rho(y)u,\;\;\forall x,y\in \g, u,v\in V.
\mlabel{eq:semipro}
\end{equation}
As is well-known, $\g\oplus V$ is a Lie algebra
if and only if $(V,\rho)$ is a representation of $\g$.
The Lie algebra is denoted by $\g\ltimes_\rho V$ and is called the {\bf semi-direct product} of $\g$ and $V$.
Similarly for an endo Lie algebra, we have

\begin{pro}
Let $(\g,\phi)$ be an endo Lie algebra, $(V, \rho)$ a
representation of the Lie algebra $\g$ and $\alpha$ a linear
operator on $V$.  Then $(\g\ltimes_{\rho}V, \phi+\alpha)$ is an
endo Lie algebra if and only if  $(V,\rho,\alpha)$ is a
representation of $(\g,\phi)$.
The resulting endo Lie algebra $(\g\ltimes_{\rho} V, \phi+\alpha)$ is called the {\bf semi-direct product} of the endo Lie algebra $(\g,\phi)$ and its representation $(V,\rho,\alpha)$.
\mlabel{pro:lhpsemi}
\end{pro}

This result follows as a special case of matched pairs of endo Lie algebras in Theorem~\mref{thm:3.1} in which the Lie algebra $\h:=V$ is the abelian Lie algebra.

We next turn our attention to dual representations of endo Lie algebras.
Denote the usual pairing between the dual space $V^*$ and $V$ by
\begin{equation}\mlabel{eq:pair}
 \langle\, ,\, \rangle : V^*\times V\mto K, ~~\langle w^*, v \rangle :=w^*(v), \quad \forall v\in V, w^*\in V^*.
\end{equation}
For a linear map $\varphi: V\mto W$, the transpose of $\varphi$ is defined by
\begin{equation}
    \mlabel{eq:trans}
\varphi^*: W^*\mto V^*, \quad \varphi^*(w^*)(v):=w^* (\varphi(v)),\;\;\forall w^*\in W^*,v\in V.
\end{equation}
For a representation $(V,\rho)$ of a Lie algebra $\g$, its {\bf dual representation} is the linear map  defined by
 \begin{equation}
\rho^*: \g\mto \End(V^*), \quad  \rho^*(x):=-\rho(x)^* ,~~
   \forall x\in \g.
\mlabel{eq:2.5}
\end{equation}

To obtain the dual representation of an
endo Lie algebra, an extra condition is needed.

\begin{lem}
\mlabel{lem:admrep}
Let $(\g, \phi)$ be an endo Lie algebra.  Let $(V, \rho)$ be a representation of the Lie algebra $\g$.
For $\beta\in \End(V)$, the triple $(V^*,\rho^*,\beta^*)$ is a representation of $(\g,\phi)$ if and only if $\beta$ satisfies
\begin{eqnarray}
\beta (\rho(\phi(x))(v))=\rho(x) (\beta(v)),\;\;\forall x\in \g, v\in V.
\mlabel{eq:dualrepn}
\end{eqnarray}
\end{lem}

\begin{proof}
By Eq.~\meqref{eq:repn}, the triple $(V^*,\rho^*,\beta^*)$ is a representation of $(\g,\phi)$ means that
$$ \beta^*(\rho^*(x))=\rho^*(\phi(x))\beta^*, \quad \forall x\in \g.$$
Then the lemma follows from Eqs.~\meqref{eq:trans}, \meqref{eq:2.5} and the nondegenerate pairing in Eq.~\meqref{eq:pair}.
\end{proof}

We reserve a notion for this property due to its pivotal role in our study.

\begin{defi}
Let $(\g, \phi)$ be an endo Lie algebra.  Let $(V, \rho)$ be a representation of the Lie algebra $\g$ and let $\beta\in \End(V)$.
We say that $\beta$ {\bf \dreping the endo Lie algebra $(\g,\phi)$ on $(V,\rho)$} if
$(V^*,\rho^*,\beta^*)$ is a representation of $(\g,\phi)$, that
is, Eq.~\meqref{eq:dualrepn} holds.
When $(V,\rho)$ is taken to be the adjoint representation $(\g,\ad)$ of the Lie algebra $\g$, we simply say that $\beta$ {\bf \dreping $(\g,\phi)$}.
\mlabel{de:admop}
\end{defi}

For later applications, we display some direct consequences.
By Lemma~\mref{lem:admrep} we have
\vspace{-.1cm}
\begin{cor}
Let $(\g, \phi)$ be an endo Lie algebra.
A linear operator
$\psi$ on $\g$ \dreping $(\g,\phi)$ if and
only if
\vspace{-.3cm}
\begin{eqnarray}
\psi[\phi(x),y]=[x,\psi(y)],\;\;\forall x,y\in \g.
 \mlabel{eq:Lieadmiss}
 \end{eqnarray}
\end{cor}

By Proposition~\mref{pro:lhpsemi} and Definition~\mref{de:admop}, we also have

\begin{cor}
Let $(\g,\phi)$ be an endo Lie algebra,
$(V,\rho)$ be a representation of $\g$ and
$\beta\in \End(V)$. If $\beta$ \dreping $(\g,\phi)$ on $(V,\rho)$, then we have the semi-direct product endo Lie algebra
$(\g\ltimes_{\rho^*}V^*,\phi+\beta^*)$.
\mlabel{co:dualsemipro}
\end{cor}
\vspace{-.4cm}
\subsection{Matched pairs of endo Lie algebras}

We first recall the concept of a matched pair of Lie algebras~\mcite{Maj,T}.
\vspace{-.1cm}
\begin{defi}
\mlabel{de:match}
A {\bf matched pair of Lie algebras}  is a
quadruple $(\g,\h,\rho,\mu)$, where $\g:=(\g, [\;,\;]_\g)$ and
$\h:=(\h, [\;,\;]_\h)$ are Lie algebras, $\rho: \g\mto
\End (\h)$ and $\mu: \h\mto \End(\g)$ are linear maps
such that
\begin{enumerate}
\item
$(\g,\mu)$ is a representation of $(\h, [\;,\;]_\h)$,
\item
$(\h, \rho)$ is a representation of $(\g, [\;,\;]_\g)$ and
\item
\mlabel{it:mat3}
the following compatibility conditions hold: for any $x,y\in \g$ and $a,b\in \h$,
 \begin{eqnarray}
 \rho(x)[a,b]_\h-[\rho(x)a,b]_\h-[a,\rho(x)b]_\h+\rho(\mu(a)x)b-\rho(\mu(b)x)a&=&0,
 \mlabel{eq:Liemp1}\\
\mu(a)[x,y]_\g-[\mu(a)x,y]_\g-[x,\mu(a)y]_\g+\mu(\rho(x)a)y-\mu(\rho(y)a)x&=&0.
\end{eqnarray}
\end{enumerate}
\end{defi}

For Lie algebras $(\g, [\;,\;]_\g)$, $(\h, [\;,\;]_\h)$ and linear maps $\rho: \g\mto
\End(\h)$, $\mu:\h\mto \End(\g)$, define a multiplication on
the direct sum $\g\oplus \h$ by
 \begin{equation}
 [x+a, y+b]_{\bowtie}:=[x,y]_\g+\mu(a)y- \mu(b)x+\rho(x)b-\rho(y)a+[a,b]_\h,
 \ \forall x,y \in \g, a, b\in \h.
 \mlabel{eq:3.7}
 \end{equation}
Then by~\mcite{T}, $(\g\oplus \h,
[\;,\;]_{\bowtie})$ is a Lie algebra if and only if
$(\g,\h,\rho,\mu)$ is a matched pair of $\g$ and $\h$. We denote
the resulting Lie algebra $(\g\oplus \h, [\;,\;]_{\bowtie})$
by $\g\bowtie_{\rho}^{\mu} \h$ or simply $\g\bowtie \h$.
Further, for any Lie algebra $\frak l$ whose underlying vector space is a
vector space direct sum of two Lie subalgebras $\g$ and $\h$, there is a
matched pair $(\g,\h,\rho,\mu)$ such that there is an
isomorphism from the resulting Lie algebra $(\g\oplus \h, [\;,\;]_{\bowtie})$ via Eq.~(\mref{eq:3.7}) to the Lie algebra $\frak l$ and the
restrictions of the isomorphism to $\g$ and $\h$ are the identity
maps.

\begin{defi}
A {\bf matched pair of endo Lie
algebras}  is a quadruple $((\g,\phi_\g),(\h,\phi_\h),\rho,\mu)$,
where $(\g, \phi_\g)$ and $(\h, \phi_\h)$ are endo Lie algebras,
$\rho: \g\mto \End (\h)$ and $\mu: \h\mto
\End(\g)$ are linear maps such that
\begin{enumerate}
\item
$(\g,\mu,\phi_\g)$ is a representation of the endo Lie algebra $(\h, \phi_\h)$,
\mlabel{it:emat1}
\item
$(\h, \rho,\phi_\h)$ is a representation of the endo Lie algebra $(\g, \phi_\g)$,
\mlabel{it:emat2}
\item
$(\g,\h,\rho,\mu)$ is a matched pair of Lie algebras.
\mlabel{it:emat3}
\end{enumerate}
\end{defi}

We have the following characterization of matched pairs of endo Lie algebras.
\begin{thm}
Let $(\g, \phi_\g)$ and $(\h, \phi_\h)$ be endo Lie algebras and let
$(\g,\h,\rho,\mu)$ be a matched pair of the Lie algebras $\g$ and
$\h$. Then the pair $(\g\bowtie \h,\phi_{\g}+\phi_{\h})$ is an endo Lie algebra if and only if $((\g, \phi_\g), (\h, \phi_\h), \rho, \mu)$ is a matched pair of endo Lie algebras.
\mlabel{thm:3.1}
 \end{thm}

As noted after Proposition~\mref{pro:lhpsemi}, the proposition follows from the theorem as the special case when the linear space $V$ is regarded as an abelian Lie algebra $\h$.
 \begin{proof}
Let $x,y\in\g$ and $a,b\in\h$. Then we have
\begin{eqnarray*}
(\phi_{\g}+\phi_{\h})([x+a,y+b]_{\bowtie}) &=&\phi_\g([x,y]_\g)+\phi_\h(\rho(x)(b))-\phi_\g(\mu(b)(x))\\
&&+\phi_\g(\mu(a)(y))-\phi_\h(\rho(y)(a))+\phi_\h([a,b]_\h),\\
~[(\phi_{\g}+\phi_{\h})(x+a),(\phi_{\g}+\phi_{\h})(y+b)]_{\bowtie}
&=&[\phi_\g(x),\phi_\g(y)]_\g+\rho(\phi_\g(x))\phi_\h(b)-
\mu(\phi_\h(b))\phi_\g(x)\\
&& +\mu(\phi_\h(a))\phi_\g(y)-\rho(\phi_\g(y))\phi_\h(a) +[\phi_\h(a),\phi_\h(b)]_\h.
\end{eqnarray*}
Note that the equality of the left hand sides means that $(\g\bowtie \h,\phi_{\g}+\phi_{\h})$ is an endo Lie algebra, while taking $x=b=0$ and then $a=y=0$ in the equality of the right hand sides yields the first two conditions for $((\g, \phi_\g), (\h, \phi_\h), \rho, \mu)$ to be a matched pair of  endo Lie algebras. Thus the conclusion follows.
\end{proof}
\vspace{-.4cm}
\subsection{Manin triples of endo Lie algebras}
We recall the concept of a Manin triple of Lie algebras. See~\mcite{CP} for details.

\begin{defi}
\mlabel{de:1.4} A bilinear form $\frakB\in(\g\ot\g)^*$ on a Lie
algebra $\g$ is called {\bf invariant} if
\begin{equation}
\mlabel{eq:1.3}
 \mathfrak{B}([x,y], z)=\mathfrak{B}(x, [y,z]),~~\forall~x, y, z\in \g.  \end{equation}
 \end{defi}

As an analog of the notion of a Frobenius (associative) algebra, a
Lie algebra with a nondegenerate symmetric invariant bilinear form
is called a {\bf quadratic Lie algebra}.

Let $(\g,[\;,\;]_\g)$ be a Lie algebra. Suppose that there is
a Lie algebra structure $[\;,\;]_{\g^*}$ on its dual space $\g^*$
and a Lie algebra structure on the vector space direct sum
$\g\oplus \g^\ast$ which contains
both $(\g,[\;,\;]_\g)$ and $(\g^*,[\;,\;]_{\g^*})$ as Lie
subalgebras. Define a bilinear form on $\g\oplus \g^*$ by
\begin{equation}
\frakB_d(x + a^* ,y + b^* ): = \langle x,b^*\rangle + \langle a^*
, y\rangle, \quad \forall a^*, b^* \in \g^* , x, y \in \g.
\mlabel{eq:3.9}
\end{equation}
If $\frakB_d$ is invariant, then $(\g\oplus \g^*,\frakB_d)$ is a
quadratic Lie algebra and the triple $(\g\oplus\g^*,\g, \g^*)$
 of Lie algebras  is called a (standard) {\bf
Manin triple of Lie algebras} associated to $(\g,[\;,\;]_\g)$ and
$(\g^*,[\;,\;]_{\g^*})$.
This Lie algebra on $\g\oplus \g^*$ comes from a
matched pair of $\g$ and $\g^*$ in Eq.~(\mref{eq:3.7}), and hence
will also be denoted by $\g\bowtie\g^*$. Indeed, we have

\begin{thm}
\mlabel{thm:frob}
$($\mcite{CP}$)$ Let $(\g,[\;,\;]_\g)$ be a Lie algebra. Suppose
that there is a Lie algebra structure $[\;,\;]_{\g^*}$ on its dual
space $\g^\ast$. Then there is a Manin triple of Lie algebras
associated to $(\g,[\;,\;]_\g)$ and $(\g^*,[\;,\;]_{\g^*})$ if and
only if $(\g,\g^*, \ad_\g^*,\ad_{\g^*}^*)$  is a matched pair of
Lie algebras.
\end{thm}

For endo Lie algebras, we give the following definition.
\begin{defi}
\mlabel{de:1.3}
A {\bf quadratic endo Lie algebra} is a triple $(\g,\phi,\frakB)$ where
$(\g,\phi)$ is an endo Lie algebra and $(\g,\frakB)$ is a quadratic Lie algebra.

Let $\widehat{\phi}:\g\mto \g$
denote the {\bf adjoint linear transformation of $\phi$} under the
nondegenerate bilinear form $\frakB$:
\begin{equation}
\mlabel{eq:adjoint}
\mathfrak{B}(\phi(x), y)=\mathfrak{B}(x, \widehat{\phi}(y)),\;\;\forall x,y\in \g.
\end{equation}
\end{defi}

\begin{pro}
Let $(\g,\phi,\frakB)$ be a quadratic endo Lie algebra. Let $\widehat
\phi$ be the adjoint of $\phi$.
Then $\widehat{\phi}$ \dreping $(\g,\phi)$. In other words, $(\g^*,
\ad^*,{\widehat{\phi}\,}^*)$ is a representation of the endo Lie algebra $(\g,\phi)$.
Furthermore, as representations of $(\g,\phi)$, $(\g, \ad, \phi)$ and $(\g^*,
\ad^*,{\widehat{\phi}\,}^*)$ are equivalent.

Conversely, let $(\g,\phi)$ be an endo Lie
algebra and $\psi\in \End(\g)$
dually represent $(\g,\phi)$. If the  representation $(\g^*,
\ad^*,\psi^*)$ of $(\g,\phi)$ is equivalent to $(\g,\ad,\phi)$,
then there exists a nondegenerate invariant bilinear form
$\frakB$ on $\g$ such that  $\widehat{\phi}=\psi$.
\mlabel{pp:frobadm}
\end{pro}

\begin{proof}
For any $x,y,z\in \g$, we have
\begin{eqnarray*}
0&=& \frakB([\phi(x),\phi(y)],z)-\frakB(\phi[x,y],z)\\
&=&\frakB(\phi(x),[\phi(y),z])-\frakB([x,y],\widehat\phi(z))\\
&=&\frakB(x,\widehat\phi([\phi(y),z]))-\frakB(x,[y,\widehat\phi(z)]).
\end{eqnarray*}
Thus
$\widehat{\phi}([\phi(y),z])=[y,\widehat\phi(z)].$ Hence $(\g^*,
\ad^*,{\widehat{\phi\,}}^*)$ is a representation of $(\g,\phi)$. Define a
linear map $\varphi:\g\mto \g^*$ by
$$\varphi(x)(y):=\frakB(x,y),\;\;\forall x,y\in \g.$$
Since the bilinear form $\frakB$ is nondegenerate, the linear map
$\varphi$ is a linear isomorphism. Moreover, for any $x,y,z\in A$,
we have
\begin{eqnarray*}
\langle \varphi (\ad(x)y),z\rangle&=&\frakB([x,y],z)=\frakB([z,x],y)=\langle \varphi (y), [z,x]\rangle
=\langle \ad^*(x)\varphi(y), z\rangle,\\
\langle \varphi (\phi(x)), y\rangle &=&\frakB(\phi(x), y)=\frakB(x, \widehat
\phi(y))=\langle \varphi(x), \widehat{\phi}(y)\rangle=\langle {\widehat{\phi}\,}^*(\varphi(x)),y\rangle.
\end{eqnarray*}
Hence $(\g, \ad, \phi)$ is equivalent
to $(\g^*, \ad^*,{\widehat{\phi\,}}^*)$ as representations of $(\g,\phi)$.

Conversely, suppose that $\varphi:\g\rightarrow \g^*$ is a
linear isomorphism giving the equivalence between $(\g, \ad,
\phi)$ and $(\g^*,\ad^*,\psi^*)$. Define a bilinear form
$\frakB$ on $\g$ by
$$\frakB(x,y):=\langle \varphi(x), y\rangle,\;\;\forall x,y\in \g.$$
Then by a similar argument as above, we show that $\frakB$ is
a nondegenerate invariant bilinear form on $\g$ such that
$\widehat{\phi}=\psi$.
\end{proof}

We now extend the notion of Manin triples from the context of Lie algebras to
that of endo Lie algebras.

 \begin{defi}   \mlabel{de:3.4}
Let $(\g, \phi)$ be an endo Lie algebra. Suppose that $(\g^*,\psi^*)$ is also an endo Lie algebra.
A {\bf Manin triple of endo Lie algebras} associated to $(\g,\phi)$ and $(\g^*,\psi^*)$ is a Manin triple $(\g\bowtie \g^*, \g,\g^*)$ of Lie algebras
  such that $(\g\bowtie \g^*,\phi+ \psi^*, \frakB_d)$ is a quadratic endo Lie algebra. We use $((\g\bowtie \g^*,\phi+\psi^*),(\g,\phi),(\g^*,\psi))$ to denote this Manin triple.
\end{defi}

\begin{lem}
Let $(\g\bowtie \g^*,\phi+ \psi^*, \frakB_d)$ be a quadratic endo Lie algebra.
\begin{enumerate}
\item The adjoint $\widehat{ \phi+\psi^*}$ of $\phi+\psi^*$ with respect to $\frakB_d$ is $\psi+ \phi^*$. Further $\psi+\phi^*$
\dreping the endo Lie algebra $(\g\bowtie \g^*, \phi+\psi^*)$.
\mlabel{it:abas1}
\item $\psi$ \dreping the endo Lie algebra $(\g, \phi)$.
\mlabel{it:abas2}
\item $\phi^*$ \dreping the endo Lie algebra $(\g^*, \psi^*)$.
\mlabel{it:abas3}
\end{enumerate}
\mlabel{lem:abas}
\end{lem}

\begin{proof}
(\mref{it:abas1}) For any $x,y\in \g, a^*,b^*\in \g^*$, we apply
Eq.~\meqref{eq:3.9} to give
\begin{eqnarray*}
\frakB_d\big((\phi+\psi^*)(x+a^*),y+b^*\big)&=&\langle \phi(x),b^*\rangle + \langle \psi^*(a^*), y\rangle\\
&=&\langle x, \phi^*(b^*)\rangle + \langle a^*, \psi(y)\rangle\\
&=&\frakB_d(x+a^*, (\psi+\phi^*)(y+b^*)).
\end{eqnarray*}
Hence the adjoint $\widehat{ \phi+\psi^*}$ of $\phi+\psi^*$ with respect to
$\frakB_d$ is $\psi+ \phi^*$. By Proposition~\mref{pp:frobadm}, for the
quadratic Lie algebra $(\g\bowtie \g^*,\phi+ \psi^*, \frakB_d)$, the linear map $\widehat{ \phi+\psi^*}=\psi+\phi^*$ \dreping $(\g\bowtie
\g^*, \phi+\psi^*)$.

\smallskip

\noindent
 (\mref{it:abas2}) By Item~(\mref{it:abas1}), $\psi+\phi^*$ \dreping $(\g\bowtie \g^*, \phi+\psi^*)$. By Eq.~(\mref{eq:Lieadmiss}), this is the case if and only if for any $x,y\in \g, a^*,b^*\in \g^*$,
 \begin{eqnarray*}
 (\psi+\phi^*)([(\phi+\psi^*)(x+a^*),y+b^*]_{\bowtie})=[x+a^*,(\psi+\phi^*)(y+b^*)]_{\bowtie}.
\end{eqnarray*}
Now taking $a^*=b^*=0$ in the above equation, we
have $\psi[\phi(x),y]_\g=[x,\psi(y)]_\g,$ that is, $\psi$ \dreping
$(\g,\phi)$.
\smallskip

\noindent
 (\mref{it:abas3}) Likewise, taking $x=y=0$ in the above equation yields
$ \phi^*[\psi^*(a^*),b^*]_{\g^*}=[a^*,\phi^*(b^*)]_{\g^*},$
that is, $\phi^*$ \dreping $(\g^*, \psi^*)$.
\end{proof}

Enriching Theorem~\mref{thm:frob} to the context of endo Lie algebras, we obtain

\begin{thm} Let $(\g,\phi)$ be an endo Lie algebra. Suppose that there is an endo Lie algebra structure $(\g^*,\psi^*)$ on its dual space $\g^\ast$. Then there is a Manin triple of
endo Lie algebras $((\g\bowtie \g^*,\phi+\psi^*), (\g,\phi),(\g^*,\psi^*))$ associated to $(\g,\phi)$ and
$(\g^*,\psi^*)$ if and only if $((\g,\phi),(\g^*,\psi^*),
\ad_\g^*,\ad^*_{\g^*})$ is a matched pair of endo Lie algebras.
\mlabel{thm:3.7}
\end{thm}

\begin{proof}
($\Longrightarrow$) By the assumption, there is a Manin triple of
endo Lie algebras $((\g\bowtie \g^*,\phi+\psi^*),
(\g,\phi),(\g^*,\psi^*))$ associated to $(\g,\phi)$ and
$(\g^*,\psi^*)$. Then in particular $(\g\bowtie \g^*,\g,\g^*)$ is
a Manin triple of Lie algebras associated to $\g$ and $\g^*$.
Hence by Theorem~\mref{thm:frob},
$(\g,\g^*,\ad_\g^*,\ad^*_{\g^*})$ is a matched pair of Lie
algebras for which the Lie algebra on $\g\oplus \g^*$ is the Lie
algebra $\g\bowtie \g^*$.
Since the homomorphism on $\g\bowtie \g^*$  is $\phi+\psi^*$,
by Lemma~\ref{lem:abas}, $(\g^*,\ad_\g^*,\psi^*)$ and $(\g,\ad_{\g^*}^*,\phi)$ are
representations of the endo Lie algebras $(\g,\phi)$ and $(\g^*,\psi^*)$ respectively. Hence $((\g,\phi),(\g^*,\psi^*),
\ad_\g^*,\ad^*_{\g^*})$ is a matched pair of endo Lie algebras.

\smallskip

\noindent ($\Longleftarrow$) If $((\g,\phi),(\g^*,\psi^*),
\ad_\g^*,\ad^*_{\g^*})$ is a matched pair of endo Lie algebras,
then $(\g,\g^*,\ad_\g^*,\ad^*_{\g^*})$ is a matched pair of Lie
algebras. Hence by Theorem~\mref{thm:frob} again, $\frakB_d$ is
invariant on $\g\bowtie \g^*$. By Theorem~\mref{thm:3.1}, the
matched pair of endo Lie algebras also equips the Lie algebra
$\g\bowtie \g^*$ with the endomorphism  $\phi+\psi^*$, giving us a
quadratic endo Lie algebra. This is exactly what we need.
\end{proof}
\vspace{-.4cm}
\subsection{Endo Lie bialgebras}

With our previous preparations, we are ready to introduce the
notion of endo Lie bialgebras, as an enrichment of the notion of
Lie bialgebras that we now recall and refer the reader to~\mcite{CP} for details.

\begin{thm}
\mlabel{thm:md}
Let $(\g,[\;,\;]_\g)$ be a Lie algebra. Suppose that there is a
Lie algebra $(\g^*, [\;,\;]_{\g^*})$ on the linear dual $\g^*$ of
$\g$. Let $\delta:\g\mto \g\ot \g$ denote the linear dual of
the multiplication $[\;,\;]_{\g^*}:\g^*\ot \g^*\mto \g^*$ on
$\g^*$. Then the quadruple $(\g,\g^*,\ad_\g^*,\ad^*_{\g^*})$ is a
matched pair of Lie algebras if and only if, for any $x,y\in \g$, we have
 \begin{equation}
 \mlabel{eq:3.14}
\delta[x,y]_\g=(\ad_\g(x)\otimes\id+\id\otimes
\ad_\g(x))\delta(y)-(\ad_\g(y)\otimes\id+\id\otimes
\ad_\g(y))\delta(x).
 \end{equation}
\end{thm}

Under our assumption of finite dimension, the Lie algebra structure $(\g^*,[\;,\;]_{\g^*})$ is equivalent to the condition that $(\g,\delta)$ is a {\bf Lie coalgebra}~\mcite{Mi} which is defined without the dimensional restriction.

\begin{defi}
\mlabel{de:lieco}
A linear space $\g$ with a linear map $\delta:\g\mto \g\ot \g$ is called a {\bf Lie coalgebra} if $\delta$ is {\bf coantisymmetric}, in the sense that $\delta=-\tau \delta$ for the flip map $\tau:\g\ot \g\mto \g\ot \g$, and satisfies the {\bf co-Jacobian identity}:
\begin{equation}
(\id +\sigma+\sigma^2)(\id \ot \delta)\delta =0,
\end{equation}
where $\sigma(x\ot y\ot z):=z\ot x\ot y$ for $x, y, z\in \g$.
\end{defi}

Combining a Lie algebra and a Lie coalgebra gives

\begin{defi}
\mlabel{de:bial}
A {\bf Lie bialgebra} is a triple $(\g,[\;,\;]_\g,\delta)$, where $\g:=(\g,[\;,\;])$ is a Lie algebra, $(\g,\delta)$ is a Lie coalgebra (that is, $(\g^*,\delta^*)$ is a Lie algebra when $\g$ is finite dimensional) and Eq.~(\mref{eq:3.14}) holds.
\end{defi}

The notion of a Lie bialgebra applies to Lie algebras of any dimension.
Under the finite-dimension condition, the notion is characterized by matched pairs of Lie algebras because of Theorem~\mref{thm:md}, and then by Manin triples of Lie algebras thanks to Theorem~\mref{thm:frob}. With the extra structure of endomorphisms, we also have the following equivalent characterizations.

\begin{thm}
\mlabel{thm:rbbial}
\mlabel{thm:rbinfbialg}
Let $(\g, [\;,\;]_\g,\phi)$ be an endo Lie algebra. Suppose that there is an endo Lie algebra $(\g^*,[\;,\;]_{\g^*},\psi^*)$ on the linear dual $\g^*$ of $\g$. Let $\delta:\g\mto \g\ot \g$ denote the linear dual
of the multiplication $[\;,\;]_{\g^*}:\g^*\ot \g^*\mto \g^*$ on $\g^*$. Then the following statements are equivalent.
 \begin{enumerate}
 \item The quadruple $((\g,\phi),(\g^*,\psi^*), \ad_\g^*,\ad^*_{\g^*})$ is a matched pair of endo Lie algebras;
\mlabel{it:rbb1}
\item \mlabel{it:rbb2}
There is a Manin triple of endo Lie algebras associated
to the endo Lie algebras $(\g, [\;,\;]_\g,\phi)$ and $(\g^*,[\;,\;]_{\g^*},\psi^*)$;
\item \mlabel{it:rbb3}
The triple $(\g,[\;,\;]_\g,\delta)$ is a
Lie bialgebra. Furthermore, the linear operators $\phi^*$ and
$\psi$ dually represent  $(\g^*,\psi^*)$ and $(\g,\phi)$
respectively.
\end{enumerate}
\end{thm}

\begin{proof}
The equivalence (\mref{it:rbb1}) $\Longleftrightarrow$ (\mref{it:rbb2}) is given in Theorem~\mref{thm:3.7}.

By definition, Item~(\mref{it:rbb1}) means that
$(\g,\g^*,\ad_\g^*,\ad_{\g^*}^*)$ is a matched pair of Lie
algebras and that the dual representation conditions in
Item~(\mref{it:rbb3}) hold. Since the matched pair condition is
equivalent to the triple $(\g,[\;,\;]_\g,\delta)$ being a Lie
bialgebra by Theorem~\mref{thm:md}, we obtain the equivalence
(\mref{it:rbb1}) $\Longleftrightarrow$ (\mref{it:rbb3}).
\end{proof}

Given the aforementioned well-known equivalent characterizations
of a Lie bialgebra by a matched pair and a Manin triple of Lie
algebras, the characterizations in Theorem~\mref{thm:rbinfbialg}
should lead to a notion of bialgebra structure for endo Lie algebras given by
Theorem~\mref{thm:rbinfbialg}. \meqref{it:rbb3}. We first analyze
the related conditions.

\begin{defi} \mlabel{lem:rbcorb}
An {\bf endo Lie coalgebra} is a Lie coalgebra $(\g,\delta)$ together with a Lie coalgebra endomorphism $\psi:\g \mto \g$, that is, a $\psi\in \End(\g)$ such that
 \begin{eqnarray}
\mlabel{eq:corb}
 (\psi\ot \psi)\delta=\delta \psi.
 \end{eqnarray}
 \end{defi}
\vspace{-.1cm}
Under the finite-dimension condition, Eq.~\meqref{eq:corb}
is equivalent to the condition that
$\psi^*:\g^*\mto \g^*$ is an
endomorphism of the Lie algebra $\g^*$.

Note that the dual representation conditions in
Theorem~\mref{thm:rbinfbialg}. (\mref{it:rbb3}) are either defined
or can be rephrased as follows without referring to the dual space
$\g^*$ and its operations:
\begin{eqnarray}
 (\id\otimes \phi)\delta&=&(\psi\otimes\id)\delta \phi,
 \mlabel{eq:pduqrd}\\
\psi[\phi(x),y]_\g&=&[x,\psi(y)]_\g, \quad \forall x, y\in \g.
 \mlabel{eq:pduql}
 \end{eqnarray}

We are thus led to the key notion of endo Lie bialgebras that applies to vector spaces of any dimensions and indeed to any modules, just like its classical counterpart of Lie bialgebras.
\vspace{-.2cm}
\begin{defi}
\mlabel{de:rbbial}
An {\bf endo Lie bialgebra} is a quintuple
$(\g,[\;,\;]_\g,\delta,\phi,\psi)$ or simply a triple
$((\g,\phi),\delta,\psi)$ in which
\begin{enumerate}
\item $(\g,[\;,\;]_\g,\delta)$ is a Lie bialgebra,
\item $(\g,[\;,\;]_\g, \phi)$ is an endo Lie algebra,
\item $(\g,\delta,\psi)$ is an endo Lie coalgebra,
\item the compatibility conditions in Eqs.~(\mref{eq:pduqrd}) -- (\mref{eq:pduql}) are
satisfied.
\end{enumerate}
\end{defi}

Returning to the finite dimensional case, we immediately have

\begin{cor}
Consider a quintuple $(\g,[\;,\;]_\g,\delta,\phi,\psi)$ where
$(\g,\phi)$ is an endo Lie algebra. %Let $[\;,\;]_{\g^*}:\g^*\ot \g^*\mto \g^*$ be the linear dual of $\delta$.
Then the quintuple is an endo Lie bialgebra if and only if anyone $($and hence all$)$ of the equivalent conditions in Theorem~\mref{thm:rbbial} is satisfied. \mlabel{co:rbb}
\end{cor}
\vspace{-.3cm}
\subsection{Homomorphisms of Lie bialgebras and Manin triples}
As both the main motivation and application of our study
of endo Lie bialgebras, we introduce a new notion of homomorphisms of
Lie bialgebras that is compatible with those of Manin triples and  matched pairs. Then
we compare this notion with the existing notion of Lie bialgebra homomorphisms.
\vspace{-.2cm}
\subsubsection{New homomorphisms for Lie bialgebras}
We can rewrite Definition~\mref{de:rbbial} in terms of morphisms of Lie bialgebras.

\begin{defi}
\mlabel{de:auto-weakhomLie}
A {\bf \weak endomorphism} on a Lie bialgebra
$(\g,[\;,\;]_\g,\delta)$ consists of a Lie algebra homomorphism
$\phi:\g\mto\g$ and a Lie coalgebra homomorphism
$\psi:\g\rightarrow \g$
satisfying Eqs.~(\mref{eq:pduqrd}) -- (\mref{eq:pduql}).
\end{defi}
Then we immediately have
\begin{pro}
The quintuple $(\g,[\;,\;]_\g,\delta,\phi,\psi)$ is an endo Lie bialgebra if and
only if $(\phi,\psi)$ is a \weak endomorphism of the Lie bialgebra
$(\g,[\;,\;]_\g,\delta)$.
\end{pro}

Definition~\mref{de:auto-weakhomLie} motivates us to give the following notion of homomorphisms between any two Lie bialgebras.

\begin{defi}
\mlabel{de:liebialghom}
Let $(\g,[\;,\;]_\g,\delta_\g)$ and $(\h,[\;,\;]_\h,\delta_\h)$ be Lie bialgebras. A {\bf \weak homomorphism of Lie bialgebras} from $(\g,[\;,\;]_\g,\delta_\g)$ to $(\h,[\;,\;]_\h,\delta_\h)$ is a pair $(\phi,\psi)$ of linear maps such that
\begin{enumerate}
\item $\phi:\g\mto \h$ is a homomorphism of Lie algebras,
\item
$\psi:\h\mto \g$ a homomorphism of Lie coalgebras,
\item the polarizations of Eq.~\meqref{eq:pduqrd} --
\meqref{eq:pduql} hold:
\begin{eqnarray}
 (\id_\g\otimes \phi)\delta_\g&=&(\psi\otimes\id_\h)\delta_\h \phi,
 \mlabel{eq:pp1}\\
\psi[\phi(x),y]_\h&=&[x,\psi(y)]_\g, \quad \forall x\in \g, y\in
\h.
 \mlabel{eq:pp2}
 \end{eqnarray}
\end{enumerate}
If both $\phi$ and $\psi$ are bijective, the pair is
called a {\bf \weak isomorphism of Lie bialgebras}. Let $\LB$
denote the category of Lie bialgebras with its morphisms thus
defined.
\end{defi}

The benefit of \weak homomorphisms of Lie bialgebras is that it is
compatible with the following naturally defined morphisms of Manin
triples derived from Manin triples of endo Lie algebras.
\vspace{-.2cm}
\begin{defi}
Let $(\g\bowtie\g^*,\g,\g^*)$ and $(\h\bowtie\h^*,\h,\h^*)$ be
Manin triples of Lie algebras. A {\bf \weak homomorphism} between
them is a Lie algebra homomorphism
$$f:\g\bowtie \g^*\rightarrow \h\bowtie\h^*$$
that restricts to Lie algebra homomorphisms
$$f|_\g:\g\mto \h,\quad  f|_{\g^*}:\g^*\mto \h^*.$$
If $f$ is bijective, it is called a {\bf \weak isomorphism of
Manin triples}.
Let $\MT$ denote the category of Manin triples with the morphisms thus defined.
\mlabel{de:whomMT}
\end{defi}

This notion is justified by the equivalence in the case of endomorphisms.

\begin{pro}
Let $(\g\bowtie\g^*,\g,\g^*)$ be a Manin triple of Lie algebras.
Then a linear map $f\in\End(\g \bowtie \g^*)$ is a \weak endomorphism of the Manin triple if and only
if $((\g\bowtie\g^*,f),(\g,f|_\g),(\g^*,f|_{\g^*}))$ is a Manin
triple of endo Lie algebras associated to $(\g,f|_\g)$ and $(\g^*,f|_{\g^*})$.
\mlabel{pp:endmt}
\end{pro}

Due to the correspondence between endo Lie bialgebras and Manin
triples of endo Lie algebras given in Theorem~\mref{thm:rbinfbialg}, we obtain

\begin{pro}
Let $(\g,[\;,\;]_\g,\delta)$ be a Lie bialgebra and
$(\g\bowtie\g^*,\g,\g^*)$ be the corresponding Manin triple. Then
$(\phi,\psi)$ is a \weak endomorphism of the Lie bialgebra
$(\g,[\;,\;]_\g,\delta)$ if and only if $\phi+\psi^*$ is a \weak
endomorphism of the Manin triple $(\g\bowtie\g^*,\g,\g^*)$.
\end{pro}

Now we show that the correspondence of Lie bialgebras with Manin
triples established by Theorems~\mref{thm:frob} and ~\mref{thm:md} gives rise to an
equivalence of the corresponding categories $\LB$ and $\MT$.

\begin{pro}
\mlabel{thm:catequiv}
Assume that all the spaces are finite dimensional.
\begin{enumerate}
\item  Let $(\g,[\;,\;]_\g,\delta_\g)$ and
$(\h,[\;,\;]_\h,\delta_\h)$ be Lie bialgebras. Let
$(\g\bowtie\g^*,\g,\g^*)$ and $(\h\bowtie\h^*,\h,\h^*)$ be the
corresponding Manin triples of Lie algebras. There is a bijection
between the set $\Hom_{\bf LB}(\g,\h)$ of \weak homomorphisms between the Lie bialgebras and the set $\Hom_{\bf MT}(\g\bowtie\g^*,\h\bowtie \h^*)$ of \weak homomorphisms between the Manin triples. The
bijection is given by sending $(\phi,\psi)$ to $f:=\phi+ \psi^*$
and sending $f$ to $(f|_\g, (f|_{\g^*})^*)$.
\mlabel{it:catequiv1}
\item
\mlabel{it:catequiv2}
The correspondence in \eqref{it:catequiv1}
gives an equivalence from the category $\LB$ of Lie bialgebras to
the category $\MT$ of Manin triples.
\end{enumerate}
\end{pro}

\begin{proof}
\meqref{it:catequiv1}. Assume that $(\phi,\psi)$ is a \weak
homomorphism of the Lie bialgebras. Let $x,y\in \g, a^*,b^*\in \g^*$.
Then for $f:=\phi+\psi^*:\g\bowtie \g^*\rightarrow \h\bowtie\h^*$ we have
{\small\begin{eqnarray*}
f([x+a^*,y+b^*]_{\bowtie} )&=& \phi([x,y]_\g)+
\phi(\ad_{\g^*}(a^*)y-\ad^*_{\g^*}(b^*)x)\\
&\mbox{}&
+\psi^*(\ad_\g^*(x)b^*-\ad_\g^*(y)a^*)+\psi^*([a^*,b^*]_{\g^*}),\\
~[f(x+a^*),f(y+b^*)]_{\bowtie}&=&[\phi(x),\phi(y)]_\h+\ad_{\h^*}^*(\psi^*(a^*))\phi(y)-\ad_{\h^*}^*(\psi^*(b^*))\phi(x)\\
&\mbox{}&+\ad_\h^*(\phi(x))\psi^*(b^*)-\ad_\h^*(\phi(y))\psi^*(a^*)+[\psi^*(a^*),\psi^*(b^*)]_{\h^*}.
\end{eqnarray*}}
Since $\phi:\g\rightarrow \h$ and $\psi^*:\g^*\mto \h^*$ are homomorphisms of Lie algebras, we have
$$\phi([x,y]_\g)=[\phi(x),\phi(y)]_\h, \quad \psi^*([a^*,b^*]_{\g^*})=[\psi^*(a^*),\psi^*(b^*)]_{\h^*}.$$
By Eq.~(\mref{eq:pp1}), we have
 $$\phi(\ad_{\g^*}(a^*)y)=\ad_{\h^*}^*(\psi^*(a^*))\phi(y), \quad  \phi(\ad^*_{\g^*}(b^*)x)=\ad_{\h^*}^*(\psi^*(b^*))\phi(x).$$
By Eq.~(\mref{eq:pp2}), we have
 $$\psi^*(\ad_\g^*(x)b^*)=\ad_\h^*(\phi(x))\psi^*(b^*),\quad
\psi^*(\ad_\g^*(y)a^*)=\ad_\h^*(\phi(y))\psi^*(a^*).$$
Therefore $f$ is a \weak homomorphism of Manin triples.

Conversely, by a similar argument, if $f$ is a \weak homomorphism of
Manin triples, then $(f|_\g, (f|_{\g^*})^*)$ is a \weak
homomorphism of the corresponding Lie bialgebras.

\meqref{it:catequiv2}. It follows from \meqref{it:catequiv1} directly.
\end{proof}

Due to the correspondence between matched pairs and Manin triples
of Lie algebras, we can define the \weak homomorphism of matched pairs
of Lie algebras directly from Definition~\mref{de:whomMT}.

\begin{defi}
Let $(\g,\g^*,\ad_\g^*,\ad^*_{\g^*})$ and
$(\h,\h^*,\ad_\h^*,\ad^*_{\h^*})$ be matched pairs of Lie
algebras. A {\bf \weak homomorphism} between them is a Lie algebra
homomorphism
$$f:\g\bowtie \g^*\rightarrow \h\bowtie\h^*$$
that restricts to Lie algebra homomorphisms
$$f|_\g:\g\mto \h,\quad  f|_{\g^*}:\g^*\mto \h^*.$$
If $f$ is bijective, it is called a {\bf \weak isomorphism  of matched pairs}. Let
$\MP$
denote the category of such matched pairs of Lie algebras with the
morphisms thus defined. \mlabel{de:whomMP}
\end{defi}

\begin{rmk}
We only define the \weak homomorphisms for matched pairs of the form
$(\g,\g^*,\ad_\g^*,\ad^*_{\g^*})$. This is enough for our purpose. A \weak homomorphism between any two matched pairs can be defined in the same way.
\end{rmk}
\vspace{-.2cm}
On the other hand, by Theorem~\mref{thm:3.1} we also have the following compatibility of \weak homomorphisms of matched pairs.

\begin{pro}
Let $(\g,\g^*,\ad_\g^*,\ad^*_{\g^*})$ be a matched pair of Lie
algebras. Then a linear map $f:\g\bowtie \g^*\rightarrow \g\bowtie
\g^*$  is a \weak endomorphism of  this matched pair of
Lie algebras if and only if $((\g,f|_\g),(\g^*,f|_{\g^*}),
\ad_\g^*,\ad_{\g^*}^*)$ is a matched pair of endo Lie algebras. \mlabel{pp:endmp}
\end{pro}

We further have
\begin{pro}
\mlabel{pro:MP}
\begin{enumerate}
\item \mlabel{it:mp1}
Let $(\g\bowtie\g^*,\g,\g^*)$ and $(\g\bowtie\h^*,\h,\h^*)$ be
Manin triples of Lie algebras and let
$(\g,\g^*,\ad_\g^*,\ad^*_{\g^*})$ and
$(\h,\h^*,\ad_\h^*,\ad^*_{\h^*})$ be the corresponding matched
pairs of Lie algebras. Then a linear map $f:\g\bowtie
\g^*\rightarrow \h\bowtie\h^*$ is a \weak homomorphism of Manin
triples if and only if $f$ is a \weak homomorphism of matched
pairs.
\item \mlabel{it:mp2}
The correspondence in \eqref{it:mp1}
gives an equivalence from the category $\MP$ of matched pairs of the form $(\g,\g^*,\ad_\g^*,\ad^*_{\g^*})$ to
the category $\MT$ of Manin triples.
\end{enumerate}
\end{pro}

\begin{proof}
\meqref{it:mp1} follows directly from Definitions~\mref{de:whomMT} and~\mref{de:whomMP}.
\smallskip

\noindent
\meqref{it:mp2}
By \meqref{it:mp1}, there is a bijection between the set of \weak
homomorphisms $f:\g\bowtie \g^*\rightarrow \h\bowtie\h^*$ of Manin
triples and the set of \weak homomorphisms $f:\g\bowtie
\g^*\rightarrow \h\bowtie\h^*$ of matched pairs by sending $f$ to
$f$ itself. This gives the desired equivalence.
\end{proof}

Combining Propositions~\mref{thm:catequiv} and \mref{pro:MP}, we the following three way equivalence of categories.

\begin{thm}\label{thm:equivalence111}
Under the finite-dimensional assumption, the following categories
are equivalent.
\begin{enumerate}
\item the category $\LB$ of Lie bialgebras; \item  the category
$\MT$ of Manin triples;
\item the category ${\bf MP}$ of matched pairs of the form $(\g,\g^*,\ad_\g^*,\ad^*_{\g^*})$.
\end{enumerate}
\mlabel{co:maplie}
\end{thm}
\vspace{-.3cm}
\subsubsection{Comparison with the existing notion of morphisms of Lie bialgebras}
We now compare the notion of \weak homomorphisms of Lie bialgebras in Definition~\mref{de:liebialghom} with the existing notions of homomorphisms   of Lie bialgebras and Manin triples~\mcite{CP}.

\begin{defi} (\mcite{CP})
A {\bf homomorphism of Lie bialgebras} from
$(\g,[\;,\;]_\g,\delta_\g)$ to $(\h,[\;,\;]_\h,\delta_\h)$ is a
linear map $f:\g\mto \h$ that is both a Lie algebra homomorphism
and a Lie coalgebra homomorphism:
$\delta_\h f=(f\otimes
f)\delta_\g.$
If $f$ is also bijective,
then $f$ is called an {\bf isomorphism of Lie
bialgebras}. \mlabel{de:hom}
\end{defi}

The following result shows that in the bijective case, our notion
of \weak homomorphisms coincides with the usual notion of
isomorphisms of Lie bialgebras.

\begin{pro}\label{pro:iso}
Let $(\g,[\;,\;]_\g,\delta_\g)$ and $(\h,[\;,\;]_\h,\delta_\h)$ be
two Lie bialgebras. Then $(\g,[\;,\;]_\g,\delta_\g)$ is isomorphic
to $(\h,[\;,\;]_\h,\delta_\h)$ if and only if there exists a Lie
algebra isomorphism $\phi:\g\mto\h$ such that $(\phi,\phi^{-1})$ is
a \weak isomorphism.
\end{pro}

\begin{proof}
Take $\psi=\phi^{-1}$. Then Eqs.~(\mref{eq:pp1}) and (\mref{eq:pp2})
hold automatically. Moreover, $\phi^*$ is an isomorphism of Lie
algebras if and only if ${\phi^{-1}}^*$ is an isomorphism of Lie
algebras.
\end{proof}
\vspace{-.2cm}
\begin{rmk}
However, in general, homomorphisms of Lie
bialgebras and \weak homomorphisms of Lie bialgebras are not
related. For example, when $\g=\h$, $f+f^*$ is usually not an endomorphism of the Lie algebra $\g\bowtie \g^*$.
\end{rmk}

Next we consider another notion of homomorphisms of
Manin triples of Lie algebras, originated from the notion of isomorphisms of Manin triples~\mcite{CP}.

\begin{defi}
    \mlabel{de:homMT}
    Let $(\g\bowtie\g^*,\g,\g^*)$ and
    $(\h\bowtie\h^*,\h,\h^*)$ be Manin triples of Lie
    algebras. A {\bf \strong homomorphism} between them is a
    Lie algebra homomorphism
    $$f:\g\bowtie \g^*\rightarrow \h\bowtie\h^*$$
    that restricts to Lie algebra homomorphisms
    $$f|_\g:\g\mto \h,\quad  f|_{\g^*}:\g^*\mto \h^*$$
    and is compatible with the bilinear forms from the Manin triples:
    \begin{equation}
        \mlabel{eq:MT2b}
        \frakB_{\g,d}(x,y)=\frakB_{\h,d}(f(x),f(y)),\;\;\forall x,y\in \g\bowtie \g^*.
    \end{equation}
\end{defi}

A bijective \strong homomorphism between two Manin
	triples is exactly the known notion of an {\bf isomorphism}
	between two Manin triples~\cite{CP}.
In general, a \strong homomorphism of Manin triples is a \weak homomorphism plus the compatibility condition in Eq.~\meqref{eq:MT2b}. This extra condition has a quite significant consequence.
\begin{pro}
    Let $(\g\bowtie\g^*,\g,\g^*)$ and $(\h\bowtie\h^*,\h,\h^*)$ be two
    Manin triples of Lie algebras. Let $(\phi,\psi)$ be a \strong homomorphism
    between them. Then
    $\psi\phi=\id.$
    In particular $\phi$ is injective and $\psi$ is surjective.
    \mlabel{pp:mthom}
\end{pro}

\begin{proof}
    By Eq.~(\mref{eq:MT2b}), we have
$\frakB_{\g,d}( x,a)=\frakB_{\h,d}(\phi(x),\psi^*(a))=\frakB_{\g,d}(
        \psi\phi(x),a)$ for all $x\in \g,a\in \g^*.
$   Hence $\psi\phi=\id$.
\end{proof}

We make the following remarks on the various notions of homomorphisms of Lie bialgebras and of Manin triples.

\begin{rmk}
\begin{enumerate}
\item  It is also known that an
isomorphism of Lie bialgebras amounts to an isomorphism of the
corresponding Manin triples.
\item In general, the homomorphisms
of Lie bialgebras in Definition~\mref{de:hom} do not correspond to
\strong homomorphisms of Manin triples in
Definition~\mref{de:homMT}. Indeed, since the underlying Lie
algebra $\g$ in a Lie bialgebra corresponds to the Lie subalgebra
in a Manin triple $(\g\bowtie\g^*,\g,\g^*)$, it is naturally
expected that the homomorphism of the Lie bialgebra is given by
$f|_\g$ in the homomorphism $f=f|_\g+f|_{\g^*}$ of the Manin
triple. Then by Proposition~\mref{pp:mthom}, the homomorphism
$f|_\g$ of Lie bialgebras will need to be injective. This is an
unusually strong restriction.
\item The
injectivity condition is due to the compatibility condition in
Eq.~\meqref{eq:MT2b}. Eliminating this condition leaves us with
the notion of \weak homomorphisms of Manin triples in
Definition~\mref{de:whomMP}. As shown in
Proposition~\mref{thm:catequiv}, this notion is compatible with the
notion of \weak homomorphisms of Lie bialgebras in
Definition~\mref{de:liebialghom}.
\end{enumerate}
\end{rmk}

To finish the discussion in this section, we compare the notion of
\weak homomorphisms  of Lie bialgebras with the
notion of weak homomorphisms introduced in \mcite{TBGS}.

\begin{defi} Let
$(\g,[\;,\;],\delta_1)$ and $(\g,[\;,\;],\delta_2)$ be Lie bialgebras. A
{\bf weak homomorphism}  from
$(\g,[\;,\;],\delta_2)$ to $(\g,[\;,\;],\delta_1)$
consists of a Lie algebra homomorphism $\phi:\g\rightarrow \g$ and a Lie coalgebra homomorphism $\psi:(\g,\delta_2)\rightarrow
(\g,\delta_1)$ such that 
\begin{equation}\psi[\phi(x),y]=[x,\psi(y)],\;\;\forall x,y\in
\g.
\mlabel{eq:req3}
\end{equation} If in addition, both $\phi$ and
$\psi$ are linear isomorphisms, then $(\phi,\psi)$ is called a
{\bf weak isomorphism} from
$(\g,[\;,\;],\delta_2)$ to $(\g,[\;,\;],\delta_1)$.
\end{defi}

Note that the above notions of weak homomorphisms and weak
isomorphisms  are defined only when the two Lie bialgebras have
the same underlying Lie algebra $\g$. Further they are mainly
available for the triangular Lie bialgebras in \cite{TBGS}, that
is, they are constructed from skew-symmetric classical $r$-matrices. In this case, Eq.~\meqref{eq:pp2} is the
same as Eq.~\meqref{eq:req3}, and Eq.~(\mref{eq:pp1}) holds
automatically, which implies that the two notions of  \weak and weak  homomorphisms
of Lie bialgebras coincide.
\vspace{-.3cm}
\section{Coboundary endo Lie bialgebras and homomorphisms of  classical  $r$-matrices}
\mlabel{sec:rmat}
\vspace{-.1cm}
In this section, we study coboundary endo Lie bialgebras and introduce the notions of \weak homomorphisms of classical
$r$-matrices. A \weak homomorphism between two classical
$r$-matrices gives a \weak homomorphism of their corresponding Lie
bialgebras.
\vspace{-.2cm}
\subsection{Coboundary endo Lie bialgebras}

Let $\g$ be a Lie algebra. For a given $r\in \g\otimes \g$, define
$\delta_r:\g\mto \g\ot\g$ by
 \begin{equation}
 \delta_r(x):=(\id\otimes \ad(x)+\ad(x)\otimes \id)(r), \quad \forall x\in \g.
 \mlabel{eq:4.1}
 \end{equation}
The following is an important construction of Lie bialgebras.
 \begin{pro} \mlabel{rmk:4.2} {\rm(\cite{CP})} Let $(\g,[\;,\;])$ be a Lie algebra and $r\in \g\otimes \g$. %If $\delta_r: \g\mto \g\otimes \g$ is given by Eq.~\meqref{eq:4.1},
 Then
$(\g,[\;,\;],\delta_r)$ is a Lie bialgebra, which is called a {\bf coboundary Lie bialgebra}, if and only if for
all $x\in \g$,
 \begin{equation}
\big(\ad(x)\otimes \id+\id\otimes\ad(x)\big)(r+\tau(r))=0,
 \mlabel{eq:4.111}
\end{equation}
 \begin{equation}\mlabel{eq:4.2}
 \big(\ad(x)\ot \id\ot \id+\id\ot\ad(x)\ot\id+\id\otimes \id\otimes \ad(x)\big)([r_{12},r_{13}]+[r_{13},r_{23}]+[r_{12},r_{23}])=0.
 \end{equation}
 Here $\tau:\g\otimes \g\rightarrow \g\otimes \g$ is the flip map
 and, writing $r=\sum_ia_i\otimes b_i$, we denote
$$
 [r_{12},r_{13}]=\sum_{i,j}[a_i,a_j]\otimes b_i\otimes b_j,
 [r_{13},r_{23}]=\sum_{i,j}a_i\otimes a_j\otimes [b_i,b_j],
 [r_{23},r_{12}]=\sum_{i,j}a_j\otimes [a_i,b_j]\otimes b_i.
$$
\end{pro}

For endo Lie bialgebras, we similarly define
\begin{defi} \mlabel{de:4.1}
	An endo Lie bialgebra $((\g,\phi), \delta, \psi)$ is called {\bf coboundary} if $\delta=\delta_r$  for some $r\in \g\ot \g$.
\end{defi}
Then we obtain
\begin{thm}
\mlabel{thm:pq} Let $(\g, \phi)$ be an endo Lie algebra and $\psi$
dually represent $(\g,\phi)$. Let  $r\in \g\otimes \g$. Then the
linear map $\delta_r$ induces an
endo Lie bialgebra $((\g, \phi), \delta_r, \psi)$ if and only if
$r$ satisfies Eqs.~\meqref{eq:4.111} and \meqref{eq:4.2}, and for
any $x\in \g$, the following conditions hold:
  \begin{eqnarray}\mlabel{eq:corbo}
(\psi\ad(x)\ot\id) (\id\ot \psi-\phi\ot \id)(r)+(\id\ot
 \psi\ad(x))(\psi\ot \id-\id\ot \phi)(r)&=&0,\\
 \mlabel{eq:P*admissible1}
 (\ad(x)\ot \id+\id\ot \ad \phi(x))(\id\ot \phi-\psi\ot
\id)(r)&=&0.
\end{eqnarray}
\end{thm}

\begin{proof} By Definition~\ref{de:rbbial} and
Proposition~\ref{rmk:4.2}, $((\g, \phi), \delta_r, \psi)$ is an
endo Lie bialgebra if and only if $r$ satisfies
Eqs.~\meqref{eq:4.111} and \meqref{eq:4.2}, and
Eqs.~\meqref{eq:corb} and ~\meqref{eq:pduqrd} hold. Set
$r=\sum_ia_i\otimes b_i$. For $x\in\g$, we have
\begin{eqnarray*}
&&(\psi\otimes \psi)\delta_r(x)-\delta_r\psi(x)\\
&=&\sum_i(\psi([x,a_i])\otimes \psi(b_i)+\psi(a_i)\otimes
\psi([x,b_i])-[\psi(x),a_i]\otimes b_i-a_i\otimes
[\psi(x),b_i])\\
&=&(\psi\ad(x)\otimes \id )(\id\otimes
\psi)(r)+(\id\otimes\psi\ad(x))(\psi\otimes\id)(r)\\
&&-\sum_i\big(\psi[x,\phi(a_i)]\otimes
b_i+a_i\otimes \psi[x,\phi(b_i)]\big)\\
&=&(\psi\ad(x)\otimes \id )(\id\otimes
\psi)(r)+(\id\otimes\psi\ad(x))(\psi\otimes\id)(r)\\
&&-(\psi\ad(x)\otimes \id)(\phi\otimes \id)(r) -(\id\otimes
\psi\ad(x))(\id\otimes \phi)(r)\\
&=&(\psi\ad(x)\ot\id) (\id\ot \psi-\phi\ot \id)(r)+(\id\ot
 \psi\ad(x))(\psi\ot \id-\id\ot \phi)(r).
\end{eqnarray*}
Thus Eq.~\meqref{eq:corb} holds if and only if
Eq.~\meqref{eq:corbo} holds.

Similarly we have
\begin{eqnarray*}
&& (\id\otimes \phi)\delta_r(x)-(\psi\otimes\id)\delta_r
\phi(x)\\
&=&\sum_i\big([x,a_i]\otimes \phi(b_i)+a_i\otimes
[\phi(x),\phi(b_i)]-\psi[\phi(x),a_i]\otimes
b_i-\psi(a_i)\otimes [\phi(x),b_i]\big)\\
&=&(\ad(x)\otimes\id)(\id\otimes \phi)(r)+(\id\otimes
\ad\phi(x))(\id\otimes \phi)(r)\\
&&-(\id\otimes \ad\phi(x))(\psi\otimes
\id)(r)-(\ad(x)\otimes\id)(\psi\otimes \id)(r)\\
&=&(\ad(x)\ot \id+\id\ot \ad \phi(x))(\id\ot \phi-\psi\ot \id)(r).
\end{eqnarray*}
Thus Eq.~\meqref{eq:pduqrd} holds if and only if
Eq.~\meqref{eq:P*admissible1} holds. Therefore the conclusion
holds.
\end{proof}

Consequently we have the following conclusion on \weak
homomorphisms of Lie bialgebras.

\begin{cor} Let $(\g,[\;,\;])$ be a Lie algebra and $\phi:\g\rightarrow \g$ be a Lie algebra
endomorphism. Let $\psi:\g\rightarrow \g$ be a linear map
satisfying Eq.~\meqref{eq:pduql}. Let $r\in \g\otimes \g$. % anddefine a linear map $\delta_r$ by Eq.~\meqref{eq:4.1}.
 Then
$(\g,[\;,\;],\delta_r)$ is a Lie bialgebra and $(\phi,\psi)$ is
a \weak endomorphism of Lie bialgebras if and only if
Eqs.~\meqref{eq:4.111}--\meqref{eq:P*admissible1} are satisfied.
 \end{cor}

As an application of Theorem~\ref{thm:pq}, we obtain the following
doubles from endo Lie bialgebras which are analogues of doubles
from Lie bialgebras.

\begin{thm} Let $((\g, \phi),\delta, \psi)$ be an endo Lie bialgebra. Let $\alpha:\g^*\mto \g^*\ot \g^*$ be the linear dual of the multiplication on $\g$. Then $((\g^*,\psi^*),-\alpha,\phi^*)$ is also an endo Lie bialgebra.
 Further there is an endo Lie  bialgebra structure on the direct sum $\g\oplus \g^*$ of
 the underlying vector spaces of $\g$ and $\g^*$ which contains the two endo Lie bialgebras as endo Lie sub-bialgebras.
 \mlabel{thm:4.7}
 \end{thm}

 \begin{proof}
 Denote the product on the Lie algebra $\g^*$ by
$[\;,\;]_{\g^*}$. By \mcite{CP}, $(\g^*,[\;,\;]_{\g^*},-\alpha)$
is a Lie bialgebra. Moreover, $\psi$ \dreping the
endo Lie algebra $(\g,\phi)$ whose algebra structure is given by
$-\alpha$ if and only if $\psi$ \dreping the endo Lie algebra
$(\g,\phi)$ whose algebra structure is given by $\alpha$.
Therefore with the fact that $\phi^*$ \dreping $(\g^*,\psi^*)$, we
obtain that $((\g^*,\psi^*),-\alpha,\phi^*)$ is an endo Lie
bialgebra.

Let $\{e_1, e_2, \cdots, e_n\}$ be a basis of $\g$ and $\{e^1,
e^2, \cdots, e^n\}$ its dual basis. Let
$r=\sum\limits^n_{i=1}e_i\otimes e^i$. Consider the Lie algebra
 $\g\bowtie \g^*$ induced by the matched pair $(\g,\g^*,\ad^*_\g,\ad^*_{\g^*})$. Define
 $$\delta_r(u)=(\id\otimes \ad_{\g\bowtie \g^*}(u)+\ad_{\g\bowtie \g^*}(u)\otimes
 \id)(r),\;\;\forall u\in \g\bowtie \g^*.$$
By Lemma~\mref{lem:abas}, $(\g\bowtie \g^*, \phi+\psi^*)$ is an
 endo Lie algebra that is \dreped by $(\psi+\phi^*)$. Hence
Eq. (\mref{eq:pduql}) holds. Since
\vspace{-.2cm}
\begin{eqnarray*}
((\phi+\psi^*)\otimes \id- \id\otimes (\psi+ \phi^*))(r)
 &=&\sum^n_{i=1}(\phi(e_i)\otimes e^i-e_i\otimes
 \phi^*(e^i))\stackrel{}{=}0,\\
((\psi+\phi^*)\otimes \id- \id\otimes (\phi+ \psi^*))(r)
 &=&\sum^n_{i=1}(\psi(e_i)\otimes e^i-e_i\otimes \psi^*(e^i))\stackrel{}{=}0,
 \end{eqnarray*}
Eqs.~(\mref{eq:corbo})--(\mref{eq:P*admissible1}) hold. By
\emph{\mcite{CP}}, we know that $r$ satisfies Eqs.
(\mref{eq:4.111}) and (\mref{eq:4.2}) and hence $ (\g\bowtie
\g^*,[\;,\;]_{\bowtie},\delta_r)$ is a Lie bialgebra containing
$(\g,[\;,\;]_\g,\delta)$ and $(\g^*,[\;,\;]_{\g^*},-\alpha)$ as
Lie sub-bialgebras. Thus $((\g\bowtie \g^*, \phi+\psi^*),
\delta_r, \psi+\phi^*)$ is an endo Lie bialgebra. It is obvious
that it contains $((\g, \phi),\delta, \psi)$ and
$((\g^*,\psi^*),-\alpha,\phi^*)$ as endo Lie sub-bialgebras.
 This completes the proof.
\end{proof}

Therefore there is the following construction of \weak
homomorphisms on the doubles of Lie  bialgebras.

\begin{cor} Let $(\g,[\;,\;],\delta)$ be a Lie bialgebra and $(\phi,\psi)$
be a \weak endomorphism. Let $\alpha:\g^*\mto \g^*\ot \g^*$ be the
linear dual of the multiplication on $\g$. Then $(\g^*,-\alpha^*)$
is a Lie bialgebra and $(\psi^*,\phi^*)$ is a \weak
endomorphism. Furthermore, there is a Lie  bialgebra structure on
the direct sum $\g\oplus \g^*$ of
 the underlying vector spaces of $\g$ and $\g^*$ which contains the two Lie bialgebras as Lie
 sub-bialgebras and $(\phi+\psi^*,\psi+\phi^*)$ is a \weak
 homomorphism.
\end{cor}
\vspace{-.3cm}
\subsection{Homomorphisms of classical $r$-matrices}
We now lift the relation between classical $r$-matrices and bialgebras to the level of categories.
Theorem~\ref{thm:pq} immediately gives

\begin{cor}
 \mlabel{cor:4.10}
Let $(\g, \phi)$ be an endo Lie algebra and $\psi$ dually
represent $(\g,\phi)$. Let  $r\in \g\otimes \g$. Then the linear
map $ \delta_r$ given by Eq.~\meqref{eq:4.1} induces an endo
Lie bialgebra $((\g, \phi), \delta_r, \psi)$ if
 Eq.~\meqref{eq:4.111} and the following equations hold:
 \begin{equation}\mlabel{eq:4.10}
 [r_{12},r_{13}]+[r_{13},r_{23}]+[r_{12},r_{23}]=0,
 \end{equation}
 \begin{equation}\mlabel{eq:4.11}
 (\phi\otimes \id-\id\otimes \psi)(r)=0,
 \end{equation}
 \begin{equation}\mlabel{eq:4.11a}
 (\psi\otimes \id-\id\otimes \phi)(r)=0.
 \end{equation}
\end{cor}

Eq.~(\mref{eq:4.10}) is just the classical Yang-Baxter equation (CYBE) in a Lie algebra and a solution of the CYBE in a Lie algebra is also called a {\bf
        classical $r$-matrix}. A Lie bialgebra $(\g,[\;,\;],\delta)$ is called
    {\bf quasi-triangular} if it is obtained from a classical $r$-matrix
    $r$ by Eq.~\meqref{eq:4.1} and is called {\bf triangular} if it is obtained from a skew-symmetric classical $r$-matrix $r$ (i.e.
    $r=-\tau(r)$).
 Moreover, it is
    straightforward to show that if $r$ is skew-symmetric, then Eq.~(\mref{eq:4.11}) holds
    if and only if Eq.~(\mref{eq:4.11a}) holds.

\begin{defi}
Let $(\g, \phi)$ be an endo Lie algebra. Let $r\in\g\otimes
\g$ and $\psi\in\End(\g) $. Then
Eq.~(\mref{eq:4.10}) with conditions given by
Eqs.~(\mref{eq:4.11}) and (\mref{eq:4.11a}) is called the {\bf
$\psi$-classical Yang-Baxter equation ($\psi$-CYBE) in $(\g,
\phi)$}. \mlabel{de:4.11}
\end{defi}

As in the case of endo Lie bialgebras, solutions of the
$\psi$-CYBE in an endo Lie algebra motivates us to the notion of
morphisms of classical $r$-matrices that is compatible with \weak
homomorphisms of Lie bialgebras.

\begin{defi}\mlabel{defi:isorr}
    Let $\g,\h$ be Lie algebras and $r_{\g},\;r_{\h}$ be
    classical $r$-matrices  in $\g$ and $\h$ respectively.
    A {\bf \weak homomorphism} from   $r_{\g}$ to $r_{\h}$ consists of a
    Lie algebra homomorphism $\phi:\g\rightarrow\h$ and a linear map $\psi:\h\mto \g$ satisfying
    \begin{eqnarray}(\psi\otimes {\rm id}_\h)(r_\h)&=&({\rm
            id}_\g\otimes \phi)(r_\g),\mlabel{eq:reqqq1}\\
        ({\rm id}_\h\otimes \psi)(r_\h)&=&(\phi\otimes {\rm id}_\g)(r_\g),\mlabel{eq:reqqq2}\\
        \psi[\phi(x),y]_\h&=&[x,\psi(y)]_\g,\;\;\forall x\in \g,y\in
        \h.\mlabel{eq:reqqq3}
    \end{eqnarray} If $\phi$ and $\psi$ are also linear isomorphisms,
    then $(\phi,\psi)$ is called a {\bf \weak isomorphism} from $r_\g$
    to $r_\h$. Let ${\bf Cr}$ denote the category of classical
    $r$-matrices  with the morphisms thus defined.
\end{defi}

Then by Definitions~\ref{de:4.11} and~\ref{defi:isorr}, we obtain
\begin{pro} \mlabel{pp:endocybe}
Let $(\g, \phi)$ be an endo Lie algebra and $\psi\in \End(\g)$
dually represent $(\g, \phi)$. Then $r\in \g\ot \g$ is a solution of the
$\psi$-CYBE in the endo Lie algebra $(\g,\phi)$ if and only if
 $r\in\g\ot\g$ is a classical
$r$-matrix and $(\phi,\psi)$ is a \weak endomorphism on $r$.
\end{pro}

Recall from~\mcite{CP} that two classical $r$-matrices $r_1$ and
$r_2$ in a Lie algebra $\g$ are said to be {\bf equivalent} if
there is a Lie algebra isomorphism $\phi:\g\mto\g$ such
that
$    (\phi\otimes \phi) (r_1)=r_2.$

\begin{cor} Let $\g$ be a Lie algebra and $r_1,\;r_2$ be
classical $r$-matrices in $\g$. Then $r_1$ is equivalent to
$r_2$ if and only if there exists a Lie algebra isomorphism
$\phi:\g\mto\g$ such that $(\phi,{\phi^{-1}})$  is a
\weak isomorphism from $r_1$ to $r_2$. \mlabel{co:r-eq}
\end{cor}

\begin{proof}
If $\phi:\g\mto\g$ is an equivalence from $r_1$ to
$r_2$, then it is straightforward to check that
$(\phi,{\phi^{-1}})$  satisfies Eqs.~(\mref{eq:reqqq1})--(\mref{eq:reqqq3}). Conversely, Eqs.~(\mref{eq:reqqq1}) implies
$(\phi\otimes \phi) (r_1)=r_2$.
\end{proof}

\begin{rmk}
When both $r_1$ and $r_2$ are skew-symmetric in the same Lie
algebra $\g$, then Eq.~(\mref{eq:reqqq1}) holds if and only if
Eq.~(\mref{eq:reqqq2}) holds. On the other hand, there is a notion
of weak homomorphism between two skew-symmetric $r$-matrices in a
Lie algebra $\g$ given in \mcite{TBGS} defined by
Eqs.~(\mref{eq:reqqq1}) and (\mref{eq:reqqq3}) only, that is,
without Eq.~(\mref{eq:reqqq2}). The two notions coincide in the
skew-symmetric case. Note that Definition~\ref{defi:isorr} is
valid without the skew-symmetric restriction and in different Lie
algebras.
\end{rmk}

Recall that for a Lie algebra $\g$ and a classical $r$-matrix $r$
of  in $\g$ satisfying Eq.~\meqref{eq:4.111},  the triple $(\g,[\;,\;],\delta_r)$ is
a quasi-triangular Lie bialgebra. On the level of categories, we obtain

\begin{thm} \mlabel{co:r12}
    Let $\g,\h$ be Lie algebras and $r_\g,r_\h$ be classical $r$-matrices in $\g$ and $\h$ respectively satisfying Eq.~\meqref{eq:4.111}.
If $(\phi, \psi)$ is a \weak homomorphism of the classical
$r$-matrices from $r_\g$ to $r_\h$, then $(\phi,\psi)$ is a \weak
homomorphism of the corresponding Lie bialgebras from
$(\g,[\;,\;]_\g,\delta_{r_\g})$ to $(\g,[\;,\;]_\h, \delta_{r_\h})$. This correspondence
defines a functor from the category ${\bf Crs}$ of classical
$r$-matrices satisfying Eq.~\meqref{eq:4.111}, as a full sub-category of ${\bf Cr}$
to the category ${\bf QTLB}$ of quasi-triangular Lie bialgebras, as a
    full sub-category of ${\bf LB}$.
\end{thm}

\begin{proof}
For the given pair $(\phi,\psi)$, an argument similar to the
proof of Theorem~\mref{thm:pq}, yields
\begin{enumerate}
\item  \mlabel{it:pq1-1} $(\psi\ot
\psi)\delta_{r_\h}=\delta_{r_\g} \psi$ if and only if for any
$x\in \h$,
 \begin{eqnarray}\mlabel{eq:corbob}
&&(\psi\ad_\h(x)\ot\id_\g) ((\id_\h\ot \psi)(r_\h)-(\phi\ot
\id_\g)(r_\g))\nonumber\\&&+(\id_\g\ot
 \psi\ad_\h(x))((\psi\ot \id_\h)(r_\h)-(\id_\g\ot \phi)(r_\g))=0.
 \end{eqnarray}
  \item
\mlabel{it:pq2-1}
 $(\id_\g\otimes \phi)\delta_{r_\g}=(\psi\otimes\id_\h)\delta_{r_\h} \phi$ if and only if for any $x\in \g$,
\begin{eqnarray}\mlabel{eq:P*admissible1b}
&&(\ad_\g(x)\ot \id_\h+\id_\g\ot \ad_\h \phi(x))((\psi\ot
\id_\h)(r_\h)-(\id_\g\ot \phi)(r_\g))=0.
 \end{eqnarray}
\end{enumerate}
Thus $(\phi,\psi)$ is a \weak homomorphism of Lie bialgebras from
$(\g,[\;,\;]_\g,\delta_{r_\g})$ to $(\g,[\;,\;]_\h, \delta_{r_\h})$. The rest of the
proof is straightforward.
\end{proof}

Specializing to the skew-symmetric case gives the conclusion
\cite[Proposition 7.17]{TBGS}, obtained by a different approach
using $\mathcal O$-operators.

\section{Homomorphisms of $\calo$-operators and pre-Lie algebras}
\mlabel{sec:oop}

We define $\mathcal O$-operators on endo Lie algebras and endo
pre-Lie algebras. 
Then an analysis similar to the previous sections motivates us to define 
\weak homomorphisms  of $\mathcal O$-operators and of pre-Lie algebras, giving rise to the categories of $\mathcal O$-operators and of pre-Lie algebras. Moreover, there are natural functors between these categories and further to the category of triangular Lie bialgebras, completing Diagram~\meqref{eq:bigdiagcat}. As in the previous sections, underneath these functors among categories, there are correspondences among the $\calo$-operator and pre-Lie algebra structures on endo Lie algebras. The discussion is similar and will not be elaborated further.  

\subsection{$\mathcal{O}$-operators on endo Lie algebras and homomorphisms of $\calo$-operators}
\mlabel{ss:roo}
For a vector space $\g$, through the isomorphism $\g\ot \g\cong
\Hom(\g^*,K)\ot \g\cong \Hom(\g^*,\g)$, any $r\in \g\otimes \g$ is
identified with a map from $\g^*$ to $\g$ which we   denote by
$r^\sharp$. Explicitly, writing $r=\sum_{i}a_i\otimes b_i$, then
\begin{equation}
    r^\sharp:\g^*\mto \g, \quad r^\sharp(a^*)=\sum_{i}\langle a^*, a_i \rangle b_i,
    ~~\forall a^*\in \g^*.
    \mlabel{eq:4.12}
\end{equation}
Note that $r$ is skew-symmetric if and only if
\begin{equation}
    \langle r^\sharp(a^*), b^*\rangle+\langle a^*,r^\sharp(b^*)\rangle=0,\;\;\forall
    a^*,b^*\in \g^*.\label{eq:skew-symmetry}
\end{equation}

Recall that an {\bf $\mathcal O$-operator} on a Lie algebra $\g$
associated to a representation $(V,\rho)$ is a linear map
$T:V\rightarrow \g$ satisfying
    \begin{equation}\mlabel{eq:4.18}
    [T(u),T(v)]=T(\rho(T(u))v-\rho(T(v))u), \;\;\forall u, v\in V.
\end{equation}

For an endo Lie algebra, the corresponding notion is
 \begin{defi}
    \mlabel{de:4.17}
Let $(\g, \phi)$ be an endo Lie algebra. Let $(V, \rho)$ be a
representation of the Lie algebra $\g$ and $\alpha:V\mto V$ be a
linear map. A linear map $T: V\mto \g$ is called an
{\bf $\mathcal{O}$-operator on $(\g,\phi)$ associated to
 $(V, \rho)$ and $\alpha$} if $T$ satisfies Eq.~\meqref{eq:4.18} and
    \begin{equation}\mlabel{eq:4.17}
        \phi T=T  \alpha.
    \end{equation}
 If in addition, $(V,\rho,\alpha)$ is a representation of
 $(\g,\phi)$, then $T$ is called an {\bf $\mathcal{O}$-operator associated to
 $(V, \rho,\alpha)$.}
\end{defi}

We have the following relationship between  $\calo$-operators for
endo Lie algebras and solutions of the CYBE in endo Lie algebras.
\begin{pro}
Let $(\g, \phi)$ be an endo Lie algebra and $\psi:\g\mto \g$ be a
linear map. Suppose that $r\in \g\otimes \g$ is skew-symmetric.
 Then $r$ is a solution of the $\psi$-\cybe in $(\g, \phi)$ if and only if $r^\sharp$ is an $\mathcal{O}$-operator associated to $(\g^*, \ad^*)$ and $\psi^*$.
If in addition, $\psi$ \dreping $(\g, \phi)$, then $r$ is a
solution of the $\psi$-\cybe in $(\g, \phi)$ if and only if
$r^\sharp$ is an $\mathcal{O}$-operator on $(\g, \phi)$ associated
    to the representation $(\g^*,\ad^*,\psi^*)$.
        \mlabel{ex:4.20}
\end{pro}

\begin{proof}
 By \mcite{Ku}, $r$ is a solution of the \cybe in the Lie algebra $\g$ if and only if
\begin{equation}
    [r^\sharp(a^*),r^\sharp(b^*)]=r^\sharp({\ad}^*(r^\sharp(a^*))b^*-\ad^*(r^\sharp(b^*))a^*), \quad \forall a^*, b^*\in \g^*,
    \mlabel{eq:4.161}
\end{equation}
that is, $r^\sharp:\g^*\mto \g$ is an $\mathcal O$-operator on
$\g$ associated to the  representation $(\g^*,\ad^*)$.

Moreover, let
 $r=\sum_ia_i\otimes b_i$ and for any $a^*\in \g^*$, we have
 $$
 r^\sharp(\psi^*(a^*))\stackrel{}{=}\sum\limits_{i=1}^{n}\langle \psi^*(a^*), a_i \rangle b_i= \sum\limits_{i=1}^{n}\langle a^*, \psi(a_i) \rangle b_i, \quad
 \phi(r^\sharp(a^*))\stackrel{}{=}\sum\limits_{i=1}^{n}\langle a^*, a_i \rangle \phi(b_i).
 $$
 So, $\phi r^\sharp=r^\sharp   \psi^*$ if and only if Eq.~\meqref{eq:4.11} holds. This completes the proof.
\end{proof}

We next show that the notion of $\calo$-operators for endo Lie
algebras naturally gives the following notion of morphisms of
$\calo$-operators for Lie algebras.
\begin{defi}
    Let $T_\g$ and $T_\h$ be $\mathcal O$-operators on Lie algebras $\g$ and $\h$ associated to representations $(V_\g,\rho_\g)$ and $(V_\h,\rho_\h)$
    respectively. A {\bf (\weak) homomorphism of $\calo$-operators} from $T_\g$ to $T_\h$ consists
    of a
    Lie algebra homomorphism  $\phi:\g\mto\h$ and a linear map $\alpha:V_\g\mto V_\h$ such that for all $x\in\g, v\in V_\g$,
    \begin{eqnarray}
        \alpha\rho_\g(x)(v)&=&\rho_\h(\phi(x))(\alpha(v)),\mlabel{defi:isocon2b}\\
        T_\h\alpha &=&\phi T_\g.\mlabel{defi:isocon1b}
    \end{eqnarray}
    In particular, if $\phi$ and $\alpha$ are  invertible,  then $(\phi,\alpha)$ is called an  {\bf isomorphism}  from $T_\g$ to
    $T_\h$. Let ${\bf OP}$ denote the category of
    $\mathcal O$-operators with the morphisms thus defined.
    \mlabel{defi:isoOb}
\end{defi}

Indeed, we immediately have
\begin{cor}
Let $(\g, \phi)$ be an endo Lie algebra. Let $(V,\rho)$ be a
representation of the Lie algebra $\g$ and $\alpha:\g\rightarrow
\g$ be a linear map. Then  $(V,\rho,\alpha)$ is a representation
of $(\g,\phi)$ and $T$ is an  $\mathcal{O}$-operator on
$(\g,\phi)$ associated to
 $(V, \rho,\alpha)$ if and only if $T$ is an $\mathcal O$-operator on the Lie algebra $\g$ associated to
the representation $(V,\rho)$ and
  $(\phi,\alpha)$ is an
 endomorphism on the $\mathcal O$-operator $T$. \mlabel{cor:homO}
\end{cor}

We now show that the notion of \weak homomorphisms of
$\calo$-operators is the correct one to be compatible with
classical $r$-matrices. 

\begin{thm} \mlabel{pp:oprm}
Let $r_\g,\;r_\h$ be skew-symmetric classical $r$-matrices in
Lie algebras $\g$ and $\h$ respectively. Let $\phi:\g\mto \h$ be a
Lie algebra homomorphism and $\psi:\h\mto \g$ be a linear map. Then
$(\phi, \psi)$ is a \weak homomorphism of classical $r$-matrices from $r_\g$
to $r_\h$ if and only if $(\phi,\psi^*)$ is a homomorphism between
the corresponding $\calo$-operators $r_\g^\sharp$ and
$r_\h^\sharp$.

This correspondence defines an equivalence from the category ${\bf
SCr}$ of skew-symmetric classical $r$-matrices, as a full
sub-category of ${\bf Cr}$ of classical $r$-matrices satisfying
Eq.~\meqref{eq:4.111}, to the category ${\bf SOP}_{\rm coad}$ of
$\mathcal O$-operators on Lie algebras associated to
      the coadjoint
      representations satisfying Eq.~\meqref{eq:skew-symmetry},
as a full subcategory of ${\bf OP}$ of $\mathcal
      O$-operators.
\end{thm}

\begin{proof}
We only need to prove the first conclusion. Let $x\in \g, y\in \h$ and $a^*\in \g^*$. Then
\begin{eqnarray*}
\langle \psi^*\ad_\g^*(x)a^*, y\rangle&=&\langle
\ad_\g^*(x)a^*,\psi(y)\rangle=-\langle
a^*,[x,\psi(y)]_\g\rangle,\\
\langle \ad_\h^*(\phi(x))(\psi^*(a^*)),y\rangle&=&-\langle
\psi^*(a^*),[\phi(x),y]_\h\rangle=-\langle
a^*,\psi[\phi(x),y]_\h\rangle.
\end{eqnarray*}
Hence $\psi^*\ad_\g^*(x)a^*=\ad_\h^*(\phi(x))(\psi^*(a^*))$ if and
only if $[x,\psi(y)]_\g$=$\psi[\phi(x),y]_\h$.

Let $a^*\in\g^*,b^*\in \h^*$. Then we have
\begin{eqnarray*}
\langle r_\h^\sharp\psi^*(a^*),b^*\rangle&=&\langle r_\h,
\psi^*(a^*)\otimes b^*\rangle=\langle r_\h,-b^*\otimes
\psi^*(a^*)\rangle
=\langle (\id_\h\otimes \psi)(r_\h),-b^*\otimes a^*\rangle,\\
\langle \phi r_\g^\sharp(a^*),b^*\rangle&=&\langle r_\g,
a^*\otimes \phi^*(b^*)\rangle=\langle r_\g,-\phi^*(b^*)\otimes
a^*\rangle=\langle (\phi\otimes \id_\g)(r_\g),-b^*\otimes
a^*\rangle.
\end{eqnarray*}
Hence $r_\h^\sharp\psi^*(a^*)=\phi r_\g^\sharp(a^*)$ if and only
if $(\psi\otimes \id_\h)(r_\h)=(\id_\g\otimes \phi)(r_\g)$, and if
and only if  $(\id_\h\otimes \psi)(r_\h)=(\phi\otimes
\id_\g)(r_\g)$. Therefore the conclusion follows.
\end{proof}

\begin{rmk}
When $\g=\h$, the above conclusion is exactly
~\cite[Proposition~7.10]{TBGS}.
\end{rmk}

\subsection{From $\mathcal O$-operators to classical $r$-matrices}
\mlabel{ss:oor}
In Section~\mref{ss:roo}, we see that a  skew-symmetric  solution of the CYBE gives an $\calo$-operator associated  to the adjoint representation. Going in the opposite direction, an $\mathcal O$-operator gives rises to a solution of
the CYBE in the semi-direct product Lie algebra  as follows.

\begin{lem}\label{lem:rt}{\rm (\cite{Bai2})}
Let $\g$ be a Lie algebra and $(V,\rho)$ be a representation. Let
$T: V\mto \g$ be a linear map which is identified as an element in
$(\g\ltimes_{\rho^*} V^*)\otimes (\g\ltimes_{\rho^*} V^*)$
 $($through ${\rm Hom}(V,\g)\cong V^*\otimes \g\subseteq (\g\ltimes_{\rho^*} V^*)\otimes (\g\ltimes_{\rho^*} V^*)$$)$.
Then
\begin{equation}
r_{\scriptscriptstyle{T}}:=T-\sigma(T)
\mlabel{eq:oor}
\end{equation}
is a skew-symmetric classical $r$-matrix
in the Lie algebra $\g\ltimes_{\rho^*} V^*$ if and only if $T$
is an $\mathcal O$-operator on $\g$ associated to $(V,\rho)$.
\end{lem}

We now lift Lemma~\ref{lem:rt} to the level of morphisms, that is, to use a \weak homomorphism of $\calo$-operators on Lie algebras to induce a \weak homomorphism of classical $r$-matrices in the corresponding semi-direct product Lie algebras. Since the latter Lie algebras are much larger, extra restraints are needed to give a well-defined correspondence.

\begin{thm}\mlabel{cor:maincon}
Let $\g$ and $\h$ be Lie algebras, and $(V_\g,\rho_\g),
(V_\h,\rho_\h)$ be representations of $\g$ and $\h$ respectively.
Let $T_\g: V_\g\to \g$ and $T_\h:V_\h\to \h$ be $\calo$-operators, and $r_{\scriptscriptstyle{T_\g}}$ and $r_{\scriptscriptstyle{T_\h}}$ be the corresponding skew-symmetric classical $r$-matrices defined in Lemma~\mref{lem:rt}.
Let $(\phi,\alpha)$ be a homomorphism of $\mathcal
O$-operators from $T_\g$  to $T_\h$. Then for linear maps $\psi:\h\to \g$ and $\beta:V_\h\to V_\g$, the pair $(\phi+ \beta^*,\psi+\alpha^*)$ is a \weak
homomorphism from $r_{\scriptscriptstyle{T_\g}}$ to $r_{\scriptscriptstyle{T_\h}}$
if and only if Eq.~\meqref{eq:reqqq3} and the following equations hold
\begin{eqnarray}
	\label{eq:ttt}T_\g\beta&=&\psi T_\h,\\
	\label{eq:hh1}
	\beta(\rho_\h(\phi(x))b)&=&\rho_\g(x)(\beta(b)),\;\;\forall x\in \g,b\in V_\h,\\
	\beta(\rho_\h(y)\alpha(a))&=&\rho_\g(\psi(y))a,\;\;\forall y\in
	\h,a\in V_\g.\label{eq:hh2}
\end{eqnarray}
\end{thm}

\begin{proof}
It is straightforward to deduce that the linear map
$\phi+\beta^*:\g\ltimes_{\rho_\g^*}V_\g^*\rightarrow
\h\ltimes_{\rho_\h^*}V_\h^*$ is a homomorphism of Lie algebras if
and only if $\phi$ is a homomorphism of Lie algebras and   Eq.
\meqref{eq:hh1} holds.
Let $\{e_1, e_2,\cdots, e_n\}$ be a basis of $V_\g$ and $\{e^1,
e^2, \cdots, e^n\}$ be its dual basis. Let $\{f_1, f_2,\cdots,
f_m\}$ be a basis of $V_\h$ and $\{f^1, f^2,\cdots, f^m\}$ be its
dual basis. Then the $2$-tensors of the $\calo$-operators $T_\g$ and $T_\h$ are
$\sum_{i=1}^n T_\g(e_i)\otimes e^i$ and $\sum_{j=1}^m T_\h(f_j)\otimes f^j$ respectively.
Hence
\vspace{-.1cm}
$$r_{\scriptscriptstyle{T_\g}}=\sum_{i=1}^n (T_\g(e_i)\otimes e^i-e^i\otimes T_\g(e_i)),\;\;r_{\scriptscriptstyle{T_\h}}=\sum_{j=1}^m (T_\h(f_j)\otimes f^j-f^j\otimes T_\h(f_j)),$$
giving
\vspace{-.3cm}
{\small
\begin{eqnarray*}
((\phi+\beta^*)\otimes {\rm id}_{\g\ltimes_{\rho_\g^*} V_\g^*})(r_{\scriptscriptstyle{T_\g}})&=&\sum_{i=1}^n(\phi  T_\g(e_i)\otimes e^i-\beta^*(e^i)\otimes T_\g(e_i)),\\
({\rm id_{\h\ltimes_{\rho_\h^*} V_\h^*}}\otimes
(\psi+\alpha^*))(r_{\scriptscriptstyle{T_\h}})&=&\sum_{j=1}^m(T_\h(f_j)\otimes
\alpha^*(f^j)-f^j\otimes \psi  T_\h(f_j)).
\end{eqnarray*}
}
Further,
{\small
\begin{eqnarray*}
\sum_{i=1}^n \beta^*(e^i)\otimes T_\g(e_i)&=&\sum_{i=1}^n\sum_{j=1}^m\langle \beta^*(e^i),f_j\rangle f^j\otimes T_\g(e_i)= \sum_{j=1}^m f^j\otimes \sum_{i=1}^n\langle e^i,\beta(f_j)\rangle T_\g(e_i)\\
&=&\sum_{j=1}^m f^j\otimes T_\g(\sum_{i=1}^n\langle \beta(f_j),
e^i\rangle e_i)=\sum_{j=1}^m f^j\otimes T_\g  \beta(f_j),
\end{eqnarray*}
}
and similarly, $\sum_{j=1}^m T_\h(f_j)\otimes \alpha^*(f^j)=
\sum_{i=1}^n T_\h  \alpha (e_i)\otimes e^i$. Then we obtain
$$((\phi+\beta^*)\otimes {\rm id}_{\g\ltimes_{\rho_\g^*} V_\g^*})(r_{\scriptscriptstyle{T_\g}})=({\rm id_{\h\ltimes_{\rho_\h^*} V_\h^*}}\otimes
(\psi+\alpha^*))(r_{\scriptscriptstyle{T_\h}})$$
 if and only if Eqs. \eqref{defi:isocon1b} and \eqref{eq:ttt} hold. %That is, \meqref{it:dd-3ii} holds if and only if \meqref{it:dd-3iii} holds. Similarly, \meqref{it:dd-3i} holds if and only if \meqref{it:dd-3iii} holds.
One similarly derives that
 $$( {\rm id}_{\g\ltimes_{\rho_\g^*} V_\g^*}\otimes(\phi+\beta^*))(r_{\scriptscriptstyle{T_\g}})=(
(\psi+\alpha^*)\otimes{\rm id_{\h\ltimes_{\rho_\h^*}
V_\h^*}})(r_{\scriptscriptstyle{T_\h}})$$
 if and only if Eqs. \eqref{defi:isocon1b} and \eqref{eq:ttt} hold.

 Finally, it is straightforward to deduce that  Eq.~\meqref{eq:reqqq3} holds (where $\g$ is replaced by $\g\ltimes_{\rho_\g^*}V_\g^*$, $\h$ by $\h\ltimes_{\rho_\h^*}V_\h^*$, $\phi$ by $\phi+\beta^*$ and
$\psi$ by $\psi+\alpha^*$) if and only if Eqs.~\meqref{eq:reqqq3},
\meqref{defi:isocon2b} and (\ref{eq:hh2}) hold. Therefore the
proof is completed.
\end{proof}

Applying Theorems~\mref{cor:maincon} and ~\ref{co:r12},  we obtain a large supply of \weak homomorphisms of Lie bialgebras.

\begin{cor} \mlabel{co:opliebial}
Under the assumption of Theorem~\mref{cor:maincon}, if the
linear maps $\phi,\alpha,\psi,\beta$ satisfy Eqs.~\meqref{eq:reqqq3} and \meqref{eq:ttt}--\meqref{eq:hh2}, then the pair $(\phi+ \beta^*,\psi+\alpha^*)$ is
a \weak homomorphism between the Lie bialgebras
    $(\g\ltimes_{\rho_\g^*} V_\g^*,\delta_{r_{\scriptscriptstyle{T_\g}}})$ and
    $(\h\ltimes_{\rho_\h^*} V_\h^*,\delta_{r_{\scriptscriptstyle{T_\h}}})$.
\end{cor}

To obtain from Theorem~\mref{cor:maincon} a functor from the category of $\calo$-operators and the category of classical $r$-matrices, there needs to be a consistent choice of $\psi$ and $\beta$.
%We have an obvious choice.

\begin{cor}\label{cor:zero}
Under the assumption of Theorem~\mref{cor:maincon}, the assignment of objects
$$ (T:V_\g\to \g) \mapsto r_{\scriptscriptstyle{T}},$$
and the assignment of morphisms
$$ \big((\phi,\alpha):T_\g\to T_\h\big) \mapsto \big((\phi,\alpha^*): r_{\scriptscriptstyle{T_\g}}\to r_{\scriptscriptstyle{T_\h}}\big)$$
define a functor from the category ${\bf OP}$ of $\mathcal O$-operators to the
category ${\bf SCr}$ of skew-symmetric classical $r$-matrices.
\end{cor}

\begin{proof}
Under the assumption of Theorem~\mref{cor:maincon}, it is obvious
that $\beta=\psi=0$ satisfies Eqs.~\meqref{eq:reqqq3} and
\meqref{eq:ttt}--\meqref{eq:hh2}. Hence the
conclusion holds.
\end{proof}

The following two examples give other choices for $\beta$
and $\psi$ in Theorem~\mref{cor:maincon}.

\begin{ex}\label{co:con}
In Theorem~\mref{cor:maincon}, take $\g=\h, V_\g=V_\h=V$ and
$\rho_\g=\rho_\h=\rho$. Further take $\beta=\pm \alpha$ and
$\psi=\pm \phi$. According to Theorem~\mref{cor:maincon}, assume
that $T:V\rightarrow \g$ is  an $\mathcal O$-operator on $\g$
associated to $(V,\rho)$ satisfying $T\alpha=\phi T$ and the
following equations
\begin{eqnarray}
&& \alpha\rho(x)=\rho(\phi(x))\alpha, \quad \mlabel{eq:semi1} \rho(x)\alpha=\alpha\rho(\phi(x)),\\ %\mlabel{eq:semi2}\\
&&
[\phi^2(x),\phi(y)]=[x,\phi(y)], \quad %\mlabel{eq:semi3}\\
\alpha\rho(x)\alpha=\rho(\phi(x)), \quad \forall x\in
\g.\mlabel{eq:semi4}
\end{eqnarray}
Then $r_{\scriptscriptstyle{T}}=T-\sigma(T)$ is a skew-symmetric classical  $r$-matrix
in the Lie algebra $\g\ltimes_{\rho^*} V^*$ and $(\phi \pm
\alpha^*,\pm \phi+\alpha^*)$ is a \weak  endomorphism. Moreover,
there is a Lie bialgebra $(\g\ltimes_{\rho^*} V^*,\delta_{r_{\scriptscriptstyle{T}}})$
and $(\phi\pm\alpha^*,\pm\phi+\alpha^*)$ is a \weak  endomorphism.
\end{ex}
\vspace{-.2cm}
\begin{ex}
Let Lie algebras $\g, \h$, representations $(V_\g,\rho_\g),
(V_\h,\rho_\h)$, linear maps $\phi, \alpha$, $T_\g, T_\h$ be as in
Theorem~\mref{cor:maincon}. Further assume that
%$T_\g, T_\h$ are $\calo$-operators and
$(\phi,\alpha)$ is an isomorphism of
$\mathcal O$-operators from $T_\g$  to $T_\h$, that is, $\phi$ and
$\alpha$ are linear bijections. For $0\ne \theta\in K$, take
$\beta=\theta \alpha^{-1}$ and $\psi=\theta \phi^{-1}$. Then
Eq.~(\mref{eq:reqqq3}) holds automatically since it is equivalent
to the fact that $\phi$ is an isomorphism of Lie algebras from
$\g$ to $\h$. Also Eqs.~\meqref{eq:ttt}, \meqref{eq:hh1} and
\meqref{eq:hh2} hold since $(\phi,\alpha)$ is an isomorphism of
$\mathcal O$-operators from $T_\g$  to $T_\h$. Thus by Theorem~\ref{cor:maincon},
$(\phi+\theta{\alpha^{-1}}^*,\theta \phi^{-1}+\alpha^*)$ is a
\weak  isomorphism   between the skew-symmetric classical
$r$-matrices $r_{\scriptscriptstyle{T_\g}}$ and $r_{\scriptscriptstyle{T_\h}}$. Furthermore,
$(\phi+\theta{\alpha^{-1}}^*,\theta \phi^{-1}+\alpha^*)$ is a
\weak  isomorphism  between the Lie bialgebras
$(\g\ltimes_{\rho_\g^*} V_\g^*,\delta_{r_{\scriptscriptstyle{T_\g}}})$ and
$(\h\ltimes_{\rho_\h^*} V_\h^*,\delta_{r_{\scriptscriptstyle{T_\g}}})$  in
Corollary~\mref{co:opliebial}.
\end{ex}
\vspace{-.2cm}
\subsection{Functors among $\calo$-operators, pre-Lie algebras and Lie bialgebras} \mlabel{sec:dend}
Now we consider the category of pre-Lie algebras and obtain an
adjoint pair of functors from it to the category of $\calo$-operators.
\vspace{-.1cm}
\begin{defi}  Let $A$ be a vector space with a bilinear product denoted by $\cdot$.
Then $(A, \cdot)$ is called a {\bf pre-Lie algebra} if
\begin{equation}
x\cdot(y\cdot z)-(x \cdot y)\cdot z=y\cdot (x\cdot z)-(y\cdot
x)\cdot z,\;\;\forall x,y,z\in A. \mlabel{eq:PreLie}
\end{equation}
A {\bf homomorphism} $f$ from a pre-Lie algebra $(A,\cdot_A)$ to
$(B,\cdot_B)$ is defined as usual, that is, $f$ is a linear map
satisfying $f(x\cdot_A y)=f(x)\cdot_B f(y)$ for all $x,y\in A$.
Let  {\bf PL}  denote the category of pre-Lie algebras with the morphisms thus defined.
\end{defi}

For any $a$ in a pre-Lie algebra $A$, let $L(a)$  denote the left
multiplication operator.
 Furthermore, define the linear map
\vspace{-.1cm}
$$L: A\mto \End(A),~a\mapsto L(a),\;\;\forall a\in A.
 $$

As is well known, for a pre-Lie algebra $(A,\cdot)$, the
multiplication
\begin{equation}
[a,b]=a\cdot b-b\cdot a,\;\; \forall a,b\in
A\mlabel{eq:5.41}\end{equation} defines a Lie algebra
$(\g(A),[\;,\;])$, called the {\bf sub-adjacent Lie algebra} of
$(A,\cdot)$. Moreover, $(A,L)$ is a representation of the Lie
algebra  $(\g(A),[\;,\;])$~\mcite{Bai2}.

\begin{defi} An {\bf \plhp} is a triple $(A,\cdot, \phi)$ consisting of a pre-Lie algebra $(A,\cdot)$ and a pre-Lie algebra endomorphism $\phi$ on $A$.\mlabel{de:5.1}
 \end{defi}
\vspace{-.2cm}

 \begin{pro} Let $(A, \cdot, \phi)$ be an \plhp.
 Then $(\g(A), [\;,\;], \phi)$ is an endo Lie algebra, where $[\;,\;]$ is given by
 Eq.~\meqref{eq:5.41}. Moreover, $(A, L, \phi)$ is a representation of the endo Lie algebra
$(\g(A), [\;,\;], \phi)$. Conversely, let $(\g(A), [\;,\;], \phi)$
be an endo Lie algebra. Suppose that there is a bilinear product
denoted by $\cdot$ such that Eq.~\meqref{eq:5.41} holds and $(A,
L, \phi)$ is a representation of $(\g(A), [\;,\;], \phi)$. Then
$(A,\cdot, \phi)$ is an endo pre-Lie algebra. \mlabel{pro:5.3}
\end{pro}
\vspace{-.1cm}
\begin{proof} It follows from
a simple checking from Eq.~(\mref{eq:repn}) and
Definition~\mref{de:5.1}.
\end{proof}

\newpage

\begin{rmk}
From the viewpoint of operads,  due to the second half part of
the above conclusion, the operad of endo pre-Lie algebras is the
bisuccessor (2-splitting) of the operad of endo Lie algebras, which is
consistent with the fact that the operad of pre-Lie algebras is
the bisuccessor of the operad of Lie algebras~\cite{BBGN}.
\end{rmk}
\vspace{-.2cm}
\begin{defi} Let $(A, \cdot, \phi)$ be an \plhp.
Let $[\;,\;]$ be the product given by Eq.~(\mref{eq:5.41}). The triple
$(\g(A),[\;,\;] , \phi)$ is called the {\bf sub-adjacent endo Lie
algebra
 of $(A, \cdot, \phi)$} and $(A, \cdot, \phi)$ is
called a {\bf compatible \plhp structure on the endo Lie algebra
$(\g(A), [\;,\;], \phi)$}. \mlabel{de:5.4}
\end{defi}

Then Proposition~\mref{pro:5.3} and Corollary~\ref{cor:homO} give
the following conclusion.

\begin{cor}
Let $(A, \cdot, \phi)$ be an \plhp. Then the identity map
$\id_A:A\rightarrow \g(A)$ on $A$ is an $\mathcal O$-operator of
the sub-adjacent endo Lie algebra $(\g(A), [\;,\;], \phi)$
associated to the representation $(A, L,
 \phi)$. Equivalently, let $A$ be a pre-Lie algebra and $\phi:A\rightarrow A$ be a
pre-Lie homomorphism. Then the identity map $\id_A$ on $A$ is an
$\mathcal O$-operator of the sub-adjacent Lie algebra $\g(A)$
associated to the representation $(A,L)$ and $(\phi,\phi)$ is an
endomorphism of $\mathcal O$-operators on $\id$.
 \mlabel{co:denid}
\end{cor}
\vspace{-.2cm}
The second half part of the above conclusion can be generalized as
follows.

\begin{pro} Let $(A,\cdot_A)$ and $(B,\cdot_B)$ be pre-Lie algebras and $\phi:A\mto B$ be a pre-Lie algebra homomorphism.
The assignments
\begin{eqnarray*}
	(A,\cdot) & \mapsto& (\id_A:A\mto \g(A)),\\
\phi& \mapsto & ((\phi,\phi): \id_A\mto \id_B)
\end{eqnarray*}
\vspace{-.1cm}
defines a functor $F$ from the category $\PL$ of pre-Lie algebras
to the category $\OP$ of $\calo$-operators. \mlabel{pro:prelieo}
\end{pro}
\begin{proof}
Since $\phi$ is also a homomorphism  between the sub-adjacent Lie
algebras and $\id_B\phi=\phi\id_A$, the pair $(\phi,\phi)$ is a
(\weak) homomorphism  between  the $\calo$-operators $\id_A$ and $\id_B$.
The other axioms of functors are easy to verify.
\end{proof}

In the other direction, by~\mcite{Bai2}, for a representation
$(V,\rho)$ of a Lie algebra $\g$ and an $\calo$-operator $T:V\mto
\g$ on $\g$ associated to $(V,\rho)$, the operation $u\cdot_T v:=
\rho(T(u))v,~ u, v\in V$, defines a pre-Lie algebra $(V,\cdot_T)$.

\begin{pro} Let $(\g,[\;,\;]_\g)$, $(\h,[\;,\;]_\h)$ be Lie algebras and $(V_\g, \rho_\g)$, $(V_\h, \rho_\h)$ be representations of $(\g,[\;,\;]_\g)$, $(\h,[\;,\;]_\h)$ respectively.
Let $T_\g: V_\g\mto \g$ be an $\mathcal{O}$-operator associated to
    $(V_\g, \rho_\g)$ and $T_\h: V_\h\mto \h$ be an
    $\mathcal{O}$-operator associated to $(V_\h, \rho_\h)$.
Suppose that $(\phi,\alpha)$ is a homomorphism of $\mathcal
O$-operators from $T_\g$ to $T_\h$. Then $\alpha$ is a
homomorphism of pre-Lie algebras from $(V_\g,\cdot_\g)$ to
$(V_\h,\cdot_\h)$. The assignments
\vspace{-.2cm}
\begin{eqnarray*}
	T_\g &\mapsto& (V_\g,\cdot_{T_\g}),\\
(\phi,\alpha)&\mapsto& \alpha,
\end{eqnarray*}
define a functor $G$ from the category $\OP$ of $\calo$-operators
to the category  {\bf PL}   of pre-Lie algebras.

Furthermore, the functor $G$ is  right adjoint to the functor $F$ in
Proposition~\mref{pro:prelieo}. \mlabel{pro:otopl}
\end{pro}
\vspace{-.2cm}
\begin{proof}
For any $u,v\in V_\g$, we have
    $$\alpha(u\cdot_{T_\g} v)=\alpha(\rho_\g(T_\g(u))v)=\rho_\h(\phi(T_\g(u)))\alpha(v)=\rho_\h(T_\h\alpha(u))\alpha(v)=\alpha(u)\cdot_{T_\h} \alpha(v).$$
    Hence $\alpha$ is a homomorphism of pre-Lie algebras from
    $(V_\g,\cdot_{T_\g})$ to $(V_\h,\cdot_{T_\h})$. The other axioms of functors are easily verified.

To prove the adjointness of the functors $F$ and $G$, we only need
to show that, for any $(A,\cdot)\in \PL$ and $T:W\mto \h$ in $\OP$,
there is a bijection
$$ \Hom_\OP(F(A,\cdot),T)\cong \Hom_\PL((A,\cdot),GT)$$
that is natural in both arguments. The left hand side consists of
pairs $(\phi,\alpha)$ with $\phi:\g(A)\mto \h$ and $\alpha:A\mto W$
such that $\phi \id_A=T\alpha$, that is, $\phi=T\alpha$, and the
right hand side consists of pre-Lie algebra homomorphisms
$\alpha:(A,\cdot)\mto (W,\cdot_T)$. Then we see that the natural
bijection can be given by sending $(\phi,\alpha)$ to $\alpha$
whose inverse is sending $\alpha$ to $(T\alpha,\alpha)$.
\end{proof}

Utilizing Proposition~\mref{pro:prelieo}, we apply homomorphisms of pre-Lie algebras to obtain three
explicitly constructed examples of \weak homomorphisms of
classical $r$-matrices and hence of Lie bialgebras.

\begin{pro}  \mlabel{co:plbialg}
Let $(A, \cdot)$ be a pre-Lie algebra. Then $r_\id$ is a
skew-symmetric classical $r$-matrix in the Lie algebra
$\g(A)\ltimes_{L^*}A^*$ and
$(\g(A)\ltimes_{L^*}A^*,\delta_{r_\id})$ is a Lie bialgebra.
Furthermore, let  $\phi:A\rightarrow A$ be an endomorphism of
pre-Lie algebras satisfying
\begin{eqnarray}\mlabel{eq:5.8ab}
(\phi^2-\id)(x)\cdot \phi(y)=\phi(x)\cdot(\phi^2-\id)(y)=0, \quad
\forall x,y\in A.
\end{eqnarray}
Then $(\phi \pm \phi^*,\pm \phi+\phi^*)$ is a \weak  endomorphism
on both the triangular $r$-matrix $r_\id$ and the Lie bialgebra
$(\g(A)\ltimes_{L^*}A^*,\delta_{r_\id})$.
\end{pro}

\begin{proof}
Note that in this case, Eq.~(\mref{eq:5.8ab})
is equivalent to Eqs.~\meqref{eq:semi1}--\meqref{eq:semi4}.
Therefore the conclusion follows from Example~\mref{co:con}.
\end{proof}

Now let $(A,\cdot_A), (B,\cdot_B)$ be pre-Lie algebras. Then
$r_{\id_A}$ and $r_{\id_B}$ are skew-symmetric classical
$r$-matrices in the Lie algebras $\g(A)\ltimes_{L^*_A}A^*$ and
$\g(B)\ltimes_{L^*_B}B^*$ respectively. Moreover,
$(\g(A)\ltimes_{L^*_A}A^*,\delta_{r_{\id_A}})$ and
$(\g(B)\ltimes_{L^*_B}B^*,\delta_{r_{\id_B}})$ are triangular Lie
bialgebras. We can thus give the following two results.

\begin{pro}
Let $\phi:(A,\cdot_A) \rightarrow (B,\cdot_B)$ be a pre-Lie
algebra homomorphism. Then $(\phi,\phi^*)$ is both a \weak
homomorphism of classical $r$-matrices from $r_{\id_A}$ to
$r_{\id_B}$ and a \weak homomorphism of Lie bialgebras from
$(\g(A)\ltimes_{L^*_A}A^*,\delta_{r_{\id_A}})$ to
$(\g(B)\ltimes_{L^*_B}B^*,\delta_{r_{\id_B}})$.
\end{pro}

\begin{proof}
The conclusion follows
from Corollaries~\ref{cor:zero} and ~\ref{co:opliebial}.
\end{proof}

\begin{pro}
Let $\phi:(A,\cdot_A) \rightarrow (B,\cdot_B)$ and
$\psi:(B,\cdot_B)\rightarrow (A,\cdot_A)$ be pre-Lie algebra
homomorphisms such that $\psi\phi=\id_A$. Then for any $0\neq
\theta\in K$, the pair $(\phi+\theta{\psi}^*,\theta \psi+\phi^*)$
is both a \weak homomorphism of skew-symmetric classical $r$-matrices from
$r_{\id_A}$ to $r_{\id_B}$ and a \weak homomorphism of Lie
bialgebras from $(\g(A)\ltimes_{L^*_A}A^*,\delta_{r_{\id_A}})$ to
$(\g(B)\ltimes_{L^*_B}B^*,\delta_{r_{\id_B}})$. In particular,  if
$\phi$ is a linear bijection, then
$(\phi+\theta{\phi^{-1}}^*,\theta \phi^{-1}+\phi^*)$ is  both a
\weak isomorphism of skew-symmetric classical $r$-matrices from $r_{\id_A}$ to
$r_{\id_B}$ and a \weak isomorphism of Lie bialgebras from
$(\g(A)\ltimes_{L^*_A}A^*,\delta_{r_{\id_A}})$ to
$(\g(B)\ltimes_{L^*_B}B^*,\delta_{r_{\id_B}})$.
\end{pro}

\begin{proof}
If $\psi$ is a homomorphism of
pre-Lie algebras and $\psi\phi=\id_A$, then $\theta \psi$
satisfies $\id_A \theta \psi=\theta \psi \id_B$,
Eqs.~\meqref{eq:reqqq3}, \meqref{eq:hh1} and \meqref{eq:hh2} where
$\psi$ is replaced by $\theta \psi$ and $\beta$ by $\theta\psi$.
Therefore the first conclusion follows from
Theorem~\ref{cor:maincon}. The special case when $\phi$ is an
isomorphism then follows directly or from
Example~\ref{co:opliebial}.
\end{proof}
\vspace{-.2cm}

\begin{rmk} Note that when $\phi$  is invertible and $\theta\ne
0$,  the inverse of $\phi+\theta{\phi^{-1}}^*$ is
$\phi^{-1}+\theta^{-1}\phi^*$. Thus by Proposition~\ref{pro:iso},
$\phi+\theta{\phi^{-1}}^*$ gives an isomorphism of Lie bialgebras
from $(\g(A)\ltimes_{L^*_A}A^*,\delta_{r_{\id_A}})$ to
$(\g(B)\ltimes_{L^*_B}B^*,\delta_{r_{\id_B}})$ if and only if
$\theta=1$. Therefore, when $\theta\ne 1$, the above isomorphism
of pre-Lie algebras provide non-trivial examples of \weak
isomorphisms of Lie bialgebras which are not the usual
isomorphisms of Lie bialgebras.
\end{rmk}

\noindent
{\bf Acknowledgments.}  This work is supported by
 National Natural Science Foundation of China (Grant Nos. 11771190, 11931009 and 11922110).   C. Bai is also
supported by the Fundamental Research Funds for the Central
Universities and Nankai ZhiDe Foundation.

\end{document}